\documentclass[a4paper,twoside,10pt]{article}

\usepackage[a4paper,left=3cm,right=3cm, top=3cm, bottom=3cm]{geometry}
\usepackage[latin1]{inputenc}
\usepackage{mathrsfs}
\usepackage{graphicx}
\usepackage{epstopdf}
\usepackage{amsmath}
\usepackage{amsthm}
\usepackage{amssymb}
\usepackage{hyperref}
\usepackage{stmaryrd}
\usepackage{float}
\usepackage{bigints}
\usepackage{cite}
\usepackage{color}
\usepackage[abs]{overpic}
\usepackage[font=footnotesize,labelfont=bf]{caption}
\usepackage{cases}
\usepackage{tikz}
\usepackage{algorithmic}
\usepackage{rotating}
\usepackage{blkarray}
\usetikzlibrary{matrix,calc,arrows}
\usepackage[author={Lorenzo}]{pdfcomment}
\usepackage{soul,xcolor}
\usepackage{enumitem}  
\usepackage{verbatim}
\usepackage{graphicx}
\usepackage{subcaption}
\usepackage{mwe}
\usepackage{algorithm}

\restylefloat{table}
\theoremstyle{plain}
\newtheorem{thm}{Theorem}[section]
\newtheorem{cor}[thm]{Corollary}
\newtheorem{lem}[thm]{Lemma}
\newtheorem{prop}[thm]{Proposition}
\theoremstyle{definition}

\theoremstyle{remark}
\newtheorem{remark}{Remark}

\newcommand{\C}{\mathbb{C}}
\newcommand{\R}{\mathbb{R}}
\newcommand{\N}{\mathbb{N}}

\newcommand{\im}{\textup{i}} 
\newcommand{\E}{K} 
\newcommand{\taun}{\mathcal T_n} 
\newcommand{\e}{e} 
\newcommand{\Pip}{\Pi_{\p}^{\E}} 
\newcommand{\Pie}{\Pi^{0,\e}_\p} 
\newcommand{\Pib}{\Pi^{0,\partial \Omega}_\p} 
\newcommand{\SE}{S^\E} 
\newcommand{\un}{u_h} 
\newcommand{\vn}{v_h} 
\newcommand{\w}{w} 
\newcommand{\zn}{z_h} 
\renewcommand{\k}{k} 
\newcommand{\n}{\mathbf n} 
\renewcommand{\Re}{\text{Re}} 
\renewcommand{\a}{a} 
\renewcommand{\b}{b} 
\newcommand{\aE}{\a^\E} 
\newcommand{\anE}{\aE_\h} 
\newcommand{\an}{\a_\h} 
\newcommand{\bn}{b_\h} 
\newcommand{\fn}{F_\h} 
\newcommand{\gn}{\g_\h} 
\newcommand{\PW}{\mathbb {PW}} 
\newcommand{\PWtaun}{\mathbb {PW}_\p(\taun)} 
\newcommand{\PWc}{\mathbb {PW}_\p^c} 
\newcommand{\PWE}{\mathbb {PW}_\p(\E)} 
\newcommand{\PWDe}{\mathbb {PW}_\p(\De)} 
\newcommand{\V}{V} 
\newcommand{\Vnp}{V_{\h}} 
\newcommand{\VE}{V_\h(\E)} 
\newcommand{\wlE}{w_\ell^{\E}} 
\newcommand{\wle}{w_\ell^e} 
\newcommand{\wje}{w_j^e} 
\newcommand{\wEDiamond}{w^{\E^{\pm}}} 
\newcommand{\dOmega}{d_\Omega} 
\newcommand{\h}{h} 
\newcommand{\hE}{\h_\E} 
\newcommand{\hDe}{\h_{\De}} 
\newcommand{\he}{\h_e} 
\newcommand{\p}{p} 
\newcommand{\pe}{p_e} 
\newcommand{\pE}{p_\E} 
\newcommand{\uPW}{w^{\taun}} 
\newcommand{\uI}{u_I} 
\newcommand{\gammaIK}{\gamma_I^K} 
\newcommand{\deltan}{\delta_\h} 
\newcommand{\Nn}{\mathcal N_\h} 
\newcommand{\g}{g} 
\newcommand{\psiE}{\psi^\E} 
\newcommand{\psiPW}{\psi^{\taun}} 
\newcommand{\psiI}{\psi_I} 
\newcommand{\Hnc}{H^{1,nc}(\taun)} 
\newcommand{\En}{\mathcal E_n} 
\newcommand{\Enb}{\mathcal E_n^B} 
\newcommand{\EnI}{\mathcal E_n^I} 
\newcommand{\EE}{\mathcal E^{\E}} 
\newcommand{\x}{\textbf{\textup{x}}} 
\newcommand{\xE}{\textbf{\textup{x}}_\E} 
\newcommand{\xe}{\textbf{\textup{x}}_{e}} 
\newcommand{\nc}{nc} 
\newcommand{\dstar}{\textbf{\textup{d}}_*} 
\newcommand{\dl}{\textbf{\textup{d}}_\ell} 
\newcommand{\djj}{\textbf{\textup{d}}_j} 
\newcommand{\cPW}{c_{PW}} 
\newcommand{\cBA}{c_{BA}} 
\newcommand{\cE}{c_{\E}} 
\newcommand{\ctilde}{\widetilde c} 
\newcommand{\ncTVEM}{nonconforming Trefftz-VEM} 
\newcommand{\ncTVE}{nonconforming Trefftz-VE} 
\newcommand{\alphah}{\alpha_\h}

\newcommand{\card}{\textup{card}} 

\newcommand{\CT}{C_T} 
\newcommand{\CP}{C_P} 
\newcommand{\CTP}{C_0} 
\newcommand{\CDelta}{C_\Delta} 
\newcommand{\ds}{\text{d}s}
\newcommand{\dx}{\text{d}x}
\newcommand{\lin}{\textup{span}}
\newcommand{\xbold}{\mathbf x} 

\newcommand{\J}{\mathcal{J}} 
\newcommand{\Je}{\mathcal{J}_e} 
\newcommand{\Jeprime}{\mathcal{J}_{e}'}

\newcommand{\dof}{\textup{dof}} 

\newcommand{\nE}{\textbf{\textup{n}}_\E} 
\newcommand{\normale}{\textbf{\textup{n}}_\e} 

\newcommand{\wE}{w^\E} 
\newcommand{\wDe}{w^{\De}} 
\newcommand{\we}{w^\e} 
\newcommand{\wb}{w^{\partial \E}} 

\newcommand{\De}{D_\e} 

\begin{tiny}
\author{
\normalsize{
}}
\end{tiny}

\date{}

\title{A nonconforming Trefftz virtual element method for the Helmholtz problem}
\date{}
\author{Lorenzo Mascotto\thanks{Faculty of Mathematics, University of Vienna, 1090 Vienna, Austria (lorenzo.mascotto@univie.ac.at, ilaria.perugia@univie.ac.at, alex.pichler@univie.ac.at)},\ Ilaria Perugia\footnotemark[1],\ 
Alexander Pichler\footnotemark[1]}

\begin{document}
\maketitle

\begin{abstract}
We introduce a novel virtual element method (VEM) for the two dimensional Helmholtz problem endowed with impedance boundary conditions.
Local approximation spaces consist of Trefftz functions,
i.e., functions belonging to the kernel of the Helmholtz operator. 
The global trial and test spaces are not fully discontinuous, but rather interelement continuity is imposed
in a nonconforming fashion. Although their functions are only implicitly defined, as typical of the VEM framework, they contain discontinuous subspaces made of functions known in closed form and with good approximation properties (plane waves, in our case).
We carry out an abstract error analysis of the method, and derive $\h$-version error estimates. Moreover, we initiate its numerical investigation by presenting a first test, which demonstrates the theoretical convergence rates.
	
	
\end{abstract}

\section{Introduction} \label{section introduction}
The virtual element method (VEM) is a recent generalization of the
finite element method (FEM) to polytopal grids  \cite{VEMvolley,
  hitchhikersguideVEM}. It
has been investigated in connection with a widespread number of
problems and engineering applications, a short list of them being
\cite{vacca2018h, BLV_StokesVEMdivergencefree, VEMchileans, hpVEMbasic, Berrone-VEM, BrennerSung_VEMsmallfaces}.
In particular, 
VEM where the continuity constraints are imposed in a nonconforming way have been the object of an extensive study
\cite{nonconformingVEMbasic, cangianimanzinisutton_VEMconformingandnonconforming, CGM_nonconformingStokes, gardini2018nonconforming, VEM_fullync_biharmonic, zhao2016nonconforming, nc_VEM_NavierStokes, hpHarmonicVEMnc}.

The main feature of VEM is that test and trial spaces consist of functions that are not known in closed form, but that are solutions to local differential problems mimicking the target one.
Despite this fact, the method is made fully computable by defining two
tools, namely suitable mappings from local approximation spaces into spaces of known functions (typically polynomials), and suitable bilinear/sesquilinear stabilization forms.

For the Helmholtz problem, a virtual version of the classical partition of unity method~\cite{BabuskaMelenk_PUMintro} was introduced in~\cite{Helmholtz-VEM}.
That method is based on discrete approximation functions given by the product of low order harmonic VE functions with plane waves.

In this work, we present a novel VE approach for the Helmholtz equation, which differs from the plane wave VEM of \cite{Helmholtz-VEM} in the two following aspects: in our method
\begin{itemize}
\item local test and trial spaces consist of functions that belong to the kernel of the Helmholtz operator;
\item no global $H^1$-conformity in the approximation space is required; instead, zero jumps of Dirichlet traces across interfaces are imposed in a nonconforming fashion.
\end{itemize}
This new method, which will be referred to as nonconforming Trefftz-VEM, does not fall into the partition of unity setting, but rather into the Trefftz one.
On the other hand, it also differs from the Trefftz methods in the Helmholtz literature, which typically employ fully discontinuous trial and test functions;
this is the case of the ultra weak variational formulation \cite{cessenatdespres_basic},
plane wave discontinuous Galerkin methods \cite{TDGPW_pversion}, 
discontinuous methods based on Lagrange multipliers \cite{farhat2001discontinuous},
wave based methods \cite{wavebasedmethod_overview}, 
the least square formulation \cite{monk1999least}, and of the variational theory of complex rays \cite{riou2008multiscale}; see \cite{PWDE_survey} for a survey.

The nonconforming approach adopted here provides an elegant
theoretical framework, where the best approximation error for
Trefftz-VE functions can be bounded in terms of the best approximation error for piecewise discontinuous Trefftz functions. 
Such property is not valid, at the best of our understanding, when employing conforming spaces.
This extends a results of \cite{hpHarmonicVEMnc}, where the best approximation error for nonconforming harmonic VE functions
was bounded by that for piecewise discontinuous harmonic polynomials.
In this sense, the nonconforming approach in the Trefftz-VE technology
provides a common framework for problems of different nature.

The \ncTVEM\,thus extends the nonconforming harmonic VEM of \cite{hpHarmonicVEMnc}, which in turn was a nonconforming version of the harmonic VEM of \cite{conformingHarmonicVEM}, to the Helmholtz problem.
It has the advantages of Trefftz methods, as it reduces considerably the number of degrees of freedom needed for achieving a given accuracy as compared to standard polynomial methods.
At the same time, it makes use of basis functions with some sort of continuity.
To be more precise, the impedance traces of the functions in the \ncTVE\,spaces 
at the boundaries of the mesh elements are prescribed to be traces of plane waves, and
the degrees of freedom are chosen to be Dirichlet moments on each edge with respect to plane waves; this allows to build global spaces with continuity of such Dirichlet moments.
In this way, information regarding the behavior of the discrete solution on the mesh skeleton can be recovered.
As compared to the partition of unity approach~\cite{BabuskaMelenk_PUMintro}, the \ncTVEM\, neither needs to have at disposal an explicit partition of unity nor requires volume quadrature formulas.
This however comes at the price of substituting the original sesquilinear form with a computable one.

In the construction of the method, we start with a larger number of degrees of freedom than for other Trefftz methods, e.g. plane wave discontinuous Galerkin method.
However, as the basis functions are associated with the mesh edges, an edgewise or\-tho\-go\-na\-li\-za\-tion-and-fil\-te\-ring process, as described in~\cite{TVEM_Helmholtz_num},
allows to significantly reduce the number of degrees of freedom without deteriorating the accuracy.
The numerical experiments presented in~\cite{TVEM_Helmholtz_num} show that the nonconforming approach becomes in this way competitive with other Trefftz methods.

\medskip\noindent
The structure of the paper is the following.
After presenting the model problem and some notation, we introduce in Section~\ref{section continuous problem} the functional setting and we describe in Section~\ref{section definition local space} the \ncTVE\,method.
An abstract error analysis and $\h$-version error estimates are
derived in Section~\ref{section a priori error analysis}. 
Finally, we present a numerical test in Section~\ref{section numerical results} and we state some conclusions in Section~\ref{section conclusions}. 

We refer to~\cite{TVEM_Helmholtz_num} a wide set of numerical experiments of the $\h$-, the $\p$-, and the $\h\p$-versions of the method, the comparison with other methods, and the description of its implementation aspects,
including the or\-tho\-go\-na\-li\-za\-tion-and-fil\-te\-ring process.

\paragraph*{Model problem.}
The model problem we consider is the following.
Given a bounded convex polygon $\Omega \subset \R^2$ with boundary $\partial \Omega$, and $\g\in H^{-\frac{1}{2}}(\partial \Omega)$, we consider the homogeneous Helmholtz problem with impedance boundary condition
\begin{equation} \label{HH continuous problem}
\left\{
\begin{alignedat}{2}
-\Delta u - \k^2 u &= 0 &&\quad \text{ in } \Omega, \\
\nabla u \cdot \textbf{\textup{n}}_\Omega + \im\k u &= \g &&\quad \text{ on } \partial \Omega, 
\end{alignedat}
\right.
\end{equation}
where $k>0$ is the wave number, $\im$ is the imaginary unit, and $\n_\Omega$ denotes the unit normal vector on $\partial \Omega$ pointing outside $\Omega$.

The corresponding variational formulation reads
\begin{equation} \label{HH problem weak formulation}
\begin{cases}
\text{find } u \in \V \text{ such that}\\
\b(u,v) = \int_{\partial \Omega} g \overline{v} \, \ds \quad  \forall v \in \V,\\
\end{cases}
\end{equation} 
where $\V:=H^1(\Omega)$ and where
\begin{equation} \label{sesquilinearform b}
\b(u,v) := \a(u,v) + \im\k \int_{\partial \Omega} u \overline{v} \, \ds\quad \forall u,\,v\in\V,
\end{equation} 
with
\begin{equation*}
\a(u,v):=\int_{\Omega} \nabla u \cdot \overline{\nabla v} \, \dx - \k^2 \int_{\Omega} u \overline{v} \, \dx\quad \forall u,\,v\in\V.
\end{equation*}
Problem~\eqref{HH problem weak formulation} is well-posed for all wave numbers $k$ and, due to the convexity assumption on $\Omega$, $u \in H^2(\Omega)$, if we assume in addition $\g \in H^{\frac{1}{2}}(\partial \Omega)$, see e.g. \cite[Proposition 8.1.4]{melenk_phdthesis}.

\paragraph*{Notation.}
We will employ the standard notation for Sobolev spaces, norms, seminorms, and sesquilinear forms with values in the complex field $\mathbb C$.
More precisely, given $s\in \mathbb N$ and $D\subset\mathbb R^d$, $d \in~\mathbb N$, $H^s(D,\mathbb C) =: H^s(D)$ denotes the space of Lebesgue measurable functions with $s$ square integrable weak derivatives.
Sobolev spaces with fractional order can be defined, e.g. via interpolation theory \cite{Triebel}.
The standard norms, seminorms, and inner products are denoted, respectively, by
\[
\Vert \cdot \Vert_{s,D},\quad\quad \vert \cdot \vert_{s,D},\quad\quad (\cdot, \cdot)_{s,D}.
\]
We highlight separately the definition of the $H^{\frac{1}{2}}(\partial D)$ inner product:
\[
(u,v)_{\frac{1}{2}, \partial D} = (u,v)_{0,\partial D} + \int_{\partial D} \int_{\partial D} \frac{(u(\boldsymbol{\xi})- u(\boldsymbol \eta))  \overline{(v(\boldsymbol \xi)- v(\boldsymbol \eta))} }{\vert \boldsymbol \xi - \boldsymbol \eta \vert^2}d \boldsymbol \xi\, d \boldsymbol\eta,
\]
where $\vert \cdot \vert$ denotes the Euclidean distance.

Moreover, given $r\in \mathbb R$, we denote by~$\mathbb R_{>r}$ and~$\mathbb R_{\ge r}$ the set of all real numbers that are greater than, and greater than or equal to~$r$, respectively;
in addition, given $m \in \mathbb N$, we denote by~$\mathbb N_{\ge m}$ the set of all natural numbers that are greater than or equal to $m$.
We define by~$B_r(\mathbf{x}_0)$ the ball centered at $\mathbf{x}_0 \in \mathbb R^2$ and with radius $r$.

Finally, we highlight that, given two positive quantities~$a$ and~$b$, we write~$a \lesssim b$ in lieu of~$a \le c\, b$ for some positive constant~$c$ independent of the discretization parameters and on the problem data.

\section{The functional setting} \label{section continuous problem}
In this section, we discuss some tools which are instrumental for the construction of the method.
More precisely, we firstly introduce the concept of regular polygonal decompositions of the physical domain in Section \ref{subsection regular polygonal decompositions}.
In Section \ref{subsection broken Sobolev spaces and pw spaces}, we recall some functional inequalities and we define broken Sobolev spaces and plane wave spaces.

\subsection{Regular polygonal decompositions and assumptions} \label{subsection regular polygonal decompositions}
We introduce here the concept of regular polygonal decompositions of the physical domain $\Omega$.

Given $\{\taun\}_{n \in \N}$ a sequence of polygonal decompositions of $\Omega$,
for every $n \in \N$ we denote by $\En$, $\EnI$ and $\Enb$ the set of edges, interior edges, and boundary edges, respectively. Moreover, for any polygon $\E \in \taun$, we denote by $\EE$ the set of its edges and we define
\begin{equation*}
\hE:=\textup{diam}(\E),  \quad n_\E:=\card(\EE).
\end{equation*}
We also define the local sesquilinear form
\begin{equation} \label{definition aE}
\aE(u,v):=\int_{\E} \nabla u \cdot \overline{\nabla v} \, \dx - \k^2 \int_{\E} u \overline{v} \, \dx \quad \forall u,v \in H^1(\E).
\end{equation}
Note that
\[
\a(u,v) = \sum_{\E \in \taun} \aE(u,v)\quad \forall u,\,v\in \V.
\]
Finally, given any $\e\in\EE$, we denote by $\he$ its length. 

For a given mesh, we define its mesh size by
\begin{equation*}
h:=\underset{\E \in \taun}{\max} \, \hE,
\end{equation*}
A sequence of polygonal decompositions $\{\taun\}_{n \in \N}$ is said to be {\it regular} if the following geometric assumptions are satisfied:
\begin{itemize}
\item[(\textbf{G1})] ({\it uniform star-shapedness}) there exist $\rho \in (0,\frac{1}{2}]$, $0<\rho_0 \le \rho$, such that, for all $n\in \mathbb N$ and for all $\E \in \taun$,
there exist points $\mathbf{x}_{0,\E}\in \E$ for which
the ball $B_{\rho \hE}(\mathbf{x}_{0,\E})$ is contained in $\E$, and $\E$ is star-shaped with respect to $B_{\rho_0 \hE}(\mathbf{x}_{0,\E})$;
\item[(\textbf{G2})] ({\it uniformly non-degenerating edges}) for all $n\in \mathbb N$ and for all $\E \in \taun$, it holds $\he \ge \rho_0 \hE$ for all edges $\e$ of $\E$, where $\rho_0$ is the same constant as in (\textbf{G1});
\item[(\textbf{G3})] ({\it uniform boundedness of the number of edges}) there exists a constant $\Lambda \in \mathbb N$ such that, for all $n\in \mathbb N$ and for all $\E \in \taun$, $n_\E \le \Lambda$, i.e., the number of edges of each element is uniformly bounded.
\end{itemize}

\subsection{Broken Sobolev spaces and plane wave spaces} \label{subsection broken Sobolev spaces and pw spaces}

Given $\E \in \taun$, we denote the impedance trace of a function $v \in H^1(K)$ on $\partial \E$ by 
\begin{equation} \label{def impedance trace}
\gammaIK(v):=\nabla v \cdot \nE +\im \k v,
\end{equation}
where $\nE$ is the unit normal vector on $\partial \E$ pointing outside $\E$.
We recall the following {\it trace inequality}, see e.g. \cite{gagliardo1957}:
\begin{align} 
\lVert v \rVert^2_{0,\partial \E} &\le \CT \left(\hE^{-1} \lVert v \rVert^2_{0,\E} + \hE \lvert v \rvert^2_{1,\E} \right) \quad \forall v \in H^1(\E) \label{trace inequality},
\end{align}
where $\CT$ depends solely on the shape of $\E$.
Further, we recall the following Poincar\'{e}-Friedrichs inequalities, see e.g. \cite{brenner2003poincare}:
\begin{align}
&\lVert \xi \rVert_{0,\E} \le \CP \hE \left( \lvert \xi \rvert_{1,\E} + \text{meas}(\Upsilon)^{-1}\left| \int_{\Upsilon} \xi \, \ds \right| \right) \quad \forall \xi \in H^1(\E), \label{Poincare Friedrich partial}\\
&\lVert \xi \rVert_{0,\E} \le \CP \hE \left( \lvert \xi \rvert_{1,\E} + \hE^{-2}\left| \int_{\E} \xi \, \dx \right| \right) \quad \forall \xi \in H^1(\E), \label{Poincare Friedrich full}
\end{align}
where~$\Upsilon$ is a measurable subset of~$\partial \E$ with 1D positive measure, and $\CP>0$ depends only on the shape of $\E$.

\begin{remark}
The constants appearing in the inequalities~\eqref{trace inequality},~\eqref{Poincare Friedrich partial}, and~\eqref{Poincare Friedrich full}, depend on the shape of the domain $\E$.
It can be proven that, if such inequalities are applied to the elements in the polygonal decompositions $\taun$ above, then these constants depend solely on
the parameters~$\rho$ and~$\Lambda$ introduced in the assumptions (\textbf{G1})-(\textbf{G2})-(\textbf{G3}), see \cite{MascottoPhDthesis} .
For the sake of simplicity, we will avoid to mention such a dependence, whenever it is clear from the context.
\end{remark}

Next, we define the broken Sobolev spaces of order $s$, for all $s\in \N$, subordinated to a decomposition $\taun$, 
\begin{equation*}
H^s(\taun):=\prod_{\E \in \taun} H^s(\E) = \{ v\in L^2(\Omega): v_{|_\E} \in H^s(\E) \quad \forall \E \in \taun \},
\end{equation*}
together with the corresponding broken Sobolev seminorms and the weighted broken Sobolev norms
\begin{equation*}
\lvert v \rvert^2_{s,\taun}:=\sum_{\E \in \taun} |v|^2_{s,\E}, \quad
\lVert v \rVert_{s,\k,\taun}^2:=\sum_{\E \in \taun} \lVert v \rVert_{s,\k,\E}^2,
\end{equation*}
respectively, where, for every $\E \in \taun$,
\begin{equation} \label{definition local weighted Sobolev norm}
\lVert v \rVert_{s,k,K}^2:=\sum_{j=0}^s \k^{2(s-j)} \lvert v \rvert_{j,K}^2.
\end{equation}
Now, we define the plane wave spaces. To this purpose, given $p=2q+1$ for some $q \in \N$, we introduce the set of indices $\J:=\{1,\dots,p\}$ and the set of pairwise different normalized directions $\{ \dl \}_{\ell \in \J}$.
For every $\E \in \taun$ and $\ell \in \J$, we define the plane wave traveling along the direction~$\dl$ as
\begin{equation} \label{plane wave bulk}
\wlE(\x):=e^{\im\k \dl \cdot (\x-\xE)}{}_{|_{\E}},
\end{equation}
and we define the local plane wave space on the element $\E \in \taun$ by
\begin{equation} \label{plane wave bulk space}
\PWE:=\lin \left\{ \wlE \, , \, \ell \in \J \right\}.
\end{equation}
We make the following assumption on the plane wave directions:
\begin{itemize}
\item[(\textbf{D1})] ({\it minimum angle}) there exists a constant $0<\delta \le 1$ with the property that the directions $\{\dl\}_{\ell \in \J}$ are such that
the minimum angle between two directions is larger than or equal to $\frac{2 \pi}{p} \delta$, and the angle between two neighbouring directions is strictly smaller than $\pi$. 
\end{itemize}


The global discontinuous plane wave space with uniform $\p$ is given by 
\begin{equation*}
\PWtaun:=\prod_{\E \in \taun} \PWE=\{  v \in L^2(\Omega) : v_{|_\E}\in \PWE \quad \forall \E\in \taun \}.
\end{equation*}
For the same $\p$, we also introduce the spaces of traces of plane waves on the mesh edges. Given $e \in \En$, we define, for any $\ell \in \J$,
\begin{equation} \label{plane wave edge}
\wle(\x):=e^{\im\k \dl \cdot (\x-\xe)}\,_{|_\e}.
\end{equation}
We observe that, while the dimension of $\PWE$ is equal to $\p$ for all $\E\in \taun$, the dimension of $\lin \{\wle,\, \ell\in \J\}$ could in principle be smaller. In fact, if
\begin{equation} \label{filter_rel}
\djj \cdot (\x-\xe) =\dl \cdot (\x - \xe) \quad \forall \x \in \e
\end{equation}
for some $j,\ell \in \{1,\dots,p\}$, $j>\ell$, then $\wje(\x)=\wle(\x)$.

Thus, we have to check for all the indices in $\J$ whether~\eqref{filter_rel} is satisfied. Whenever this is the case, we remove, without loss of generality, the index $j$ from $\J$.
This procedure is defined as the \emph{filtering process}.
We denote by $\Jeprime$ the resulting set of indices after the \emph{filtering process}. Clearly, it holds $\card(\Jeprime) \le p$.

In addition to plane waves, in the forthcoming analysis (see Section \ref{section a priori error analysis}) we will need to employ constant functions on the edges. 
To this purpose, we observe that, if there exists a direction $\dstar  \in \{\dl\}_{\ell \in \J}$ such that
\begin{equation} \label{condition adding constant}
\dstar \cdot (\x-\xe) =0 \quad \forall \x \in \e,
\end{equation}
that is, $\dstar$ is orthogonal to the edge $\e$, then $\lin\{\wle,\, \ell\in \Jeprime\}$ already contains the constant functions. Therefore, we proceed as follows.

If there exists an index $t \in \Jeprime$ such that $\mathbf{d}_t$ fulfils~\eqref{condition adding constant}, we simply set $\Je:=\Jeprime$; otherwise, we define $\Je:=\Jeprime \cup \{ p+1 \}$ and set $w_{p+1}^e(\x):=1$. Finally, we set $\pe:=\card(\Je)$.

With these definitions of the set of indices $\Je$ and of the corresponding functions $\wle$ on each edge $e \in \En$, we define the plane wave trace space of dimension $\pe$ as
\begin{equation} \label{edge space with constant}
\PWc(e):=\lin \left\{ \wle \, , \, \ell \in \Je \right\},
\end{equation}
where the superscript $c$ indicates that the space includes the constants.

We denote the space of piecewise discontinuous traces over $\partial \E$ as
\begin{equation} \label{boundary space with constants}
\PWc(\partial \E) = \{\wb\in L^2(\partial \E) : \wb{}_{|_{\e}} \in \PWc(e) \quad \forall \e\in \EE\}.
\end{equation}
In Figure \ref{fig:directions after filtering}, we depict the \emph{filtering process} applied to all the possible configurations along a given edge $\e \in \EE$.
\begin{figure}[h]
\centering
\begin{subfigure}[b]{0.4\textwidth}
\centering
\includegraphics[width=0.8\textwidth]{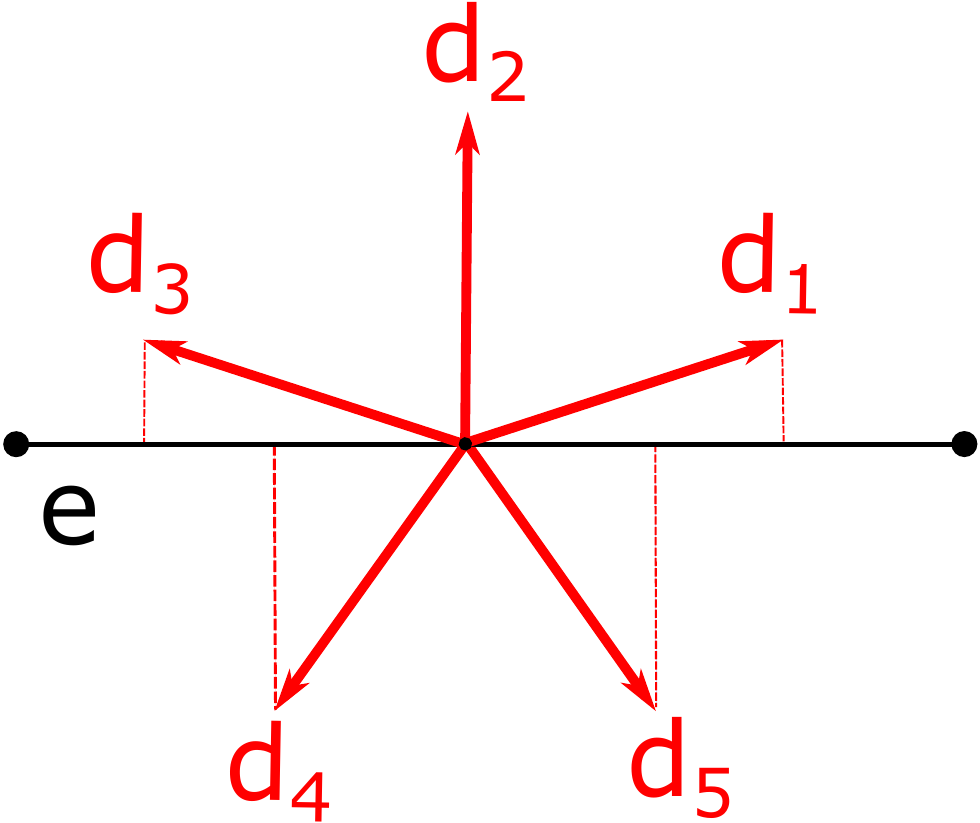}
\caption[]
{{\small No direction eliminated, orthogonal direction already included.}}    
\end{subfigure}
\hfill
\begin{subfigure}[b]{0.4\textwidth}  
\centering 
\includegraphics[width=0.8\textwidth]{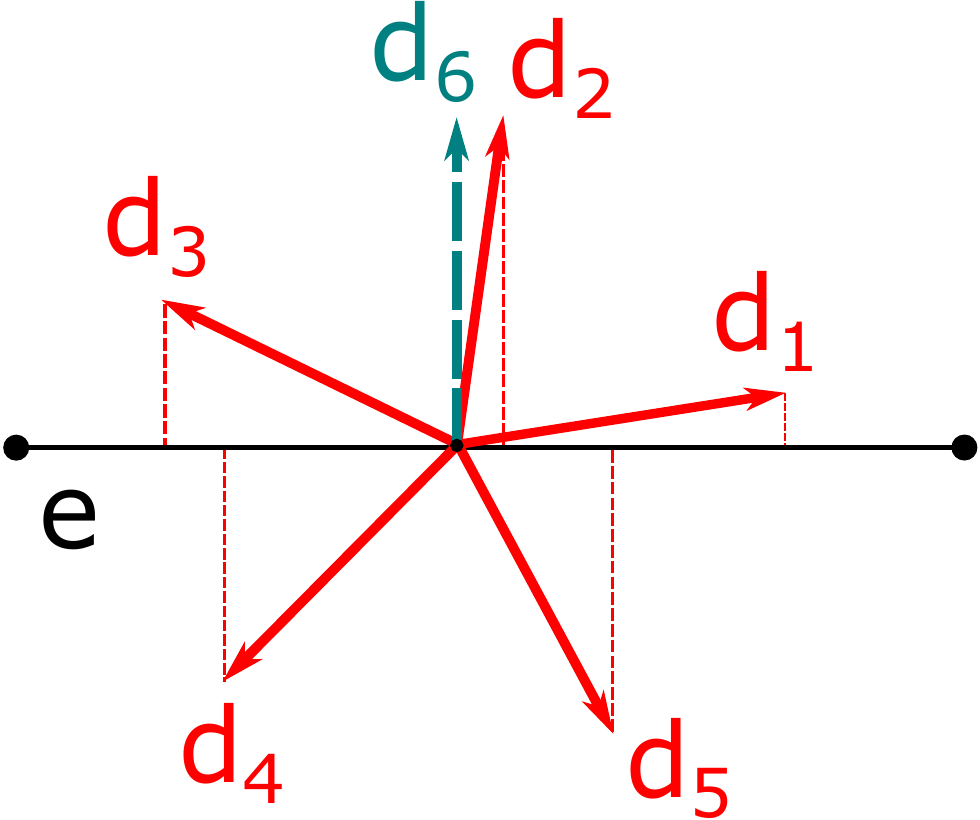}
\caption[]%
{{\small No direction eliminated, orthogonal direction not yet included.}}    
\end{subfigure}
\vskip\baselineskip
\begin{subfigure}[b]{0.4\textwidth}   
\centering 
\includegraphics[width=0.8\textwidth]{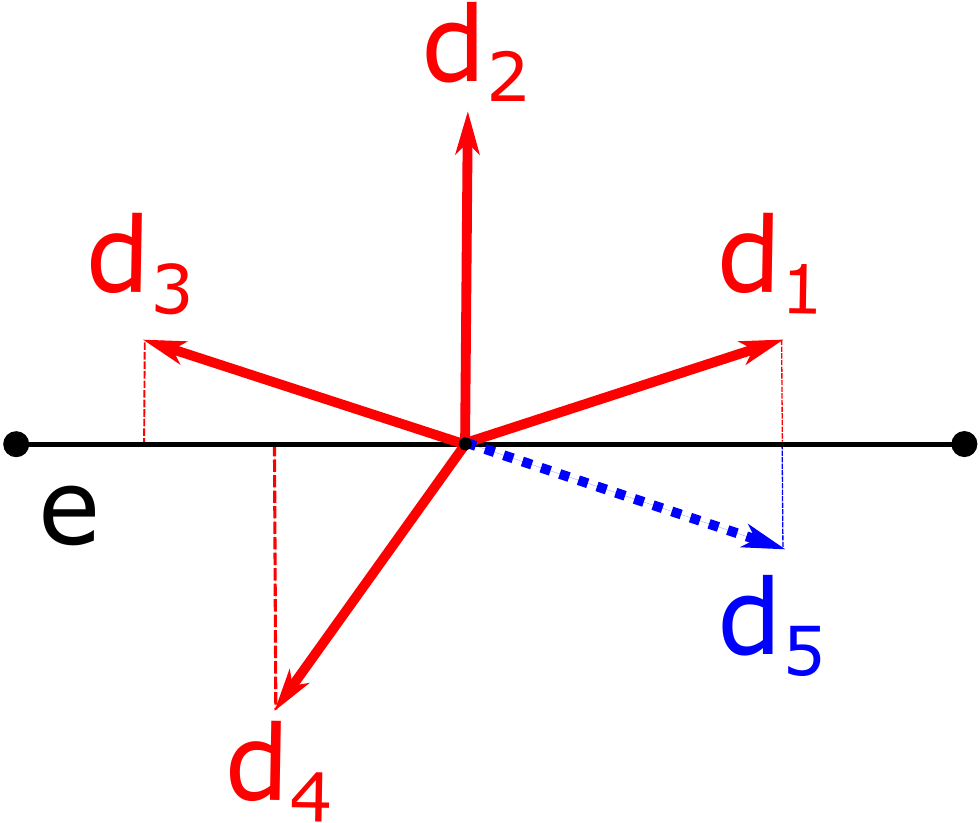}
\caption[]%
{{\small One direction eliminated, orthogonal direction already included.}}    
\end{subfigure}
\hfill
\begin{subfigure}[b]{0.4\textwidth}   
\centering 
\includegraphics[width=0.8\textwidth]{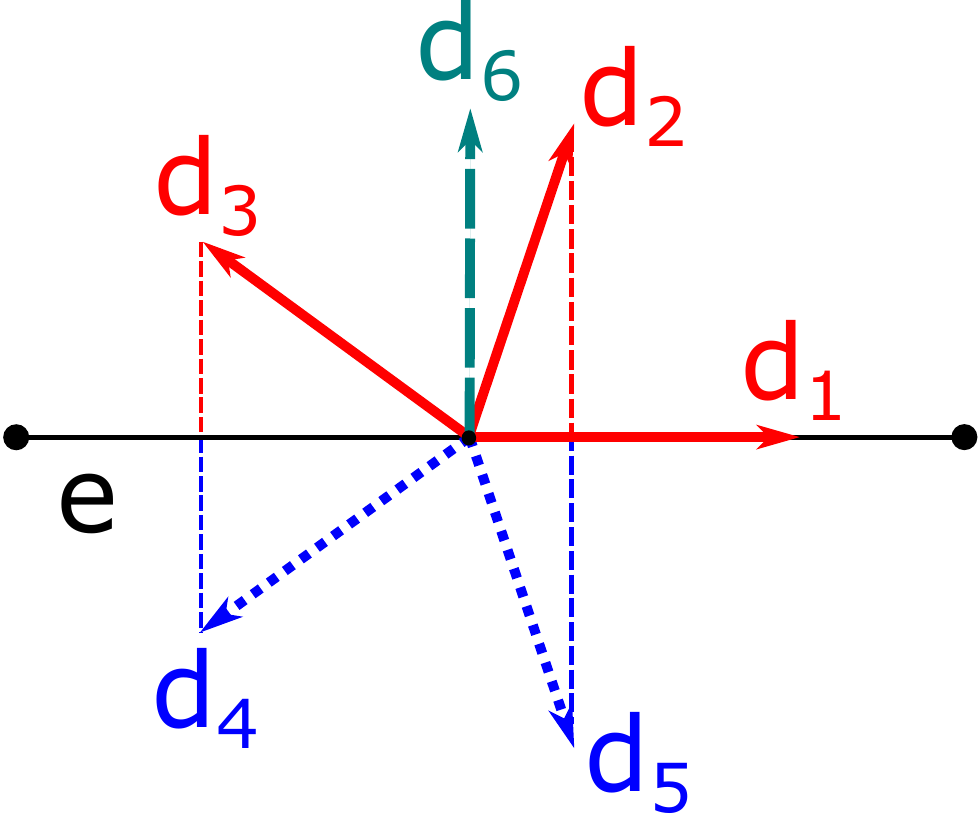}
\caption[]%
{{\small Two directions eliminated, orthogonal direction not yet included.}}    
\end{subfigure}
\caption[The average and standard deviation of critical parameters]
{\small \emph{Filtering process}.
We depict all the possible configurations. In solid lines, the directions that are kept; in dotted lines, the directions that are eliminated accordingly with~\eqref{filter_rel}; in dashed lines, the orthogonal direction that has to be possibly added in order to include constants.}
\label{fig:directions after filtering}
\end{figure}

We finally introduce the global nonconforming Sobolev space $H^{1,\nc}(\taun)$. To this end, we need to fix some additional notation.
In particular, given any internal edge $\e \in \mathcal E_n^I$ shared by the polygons $\E^-$ and $\E^+$ in $\taun$,
we denote by $\n_{\E^\pm}^e$ the two outer normal unit vectors with respect to $\partial \E^{\pm}$. For the sake of simplicity, when no confusion occurs, we will write $\n_{\E^\pm}$ in lieu of $\n_{\E^\pm}^e$.
Having this, for any $v \in H^1(\taun)$, we define the jump operator across an edge $\e \in \mathcal E^I_n$ as
\begin{equation*} 
\llbracket v \rrbracket :=
v_{|_{\E^+}} \n_{\E^+} + v_{|_{\E^-}} \n_{\E^-}.
\end{equation*}
Notice that $\llbracket v \rrbracket$ is vector-valued.

The global nonconforming Sobolev space with respect to the decomposition $\taun$ and underlying plane wave spaces with parameter~$\p \in \mathbb N$, reads
\begin{equation} \label{global non conforming space}
H^{1,nc}(\taun):=\left\{ v \in H^1(\taun): \, \int_{e} \llbracket v \rrbracket \cdot \n^e \, \overline{\we} \, \ds =0 \, \quad \forall \we \in \PWc(e), \, \forall e \in \EnI \right\},
\end{equation}
where $\n^e$ is either $\n_{\E^+}^e$ or $\n_{\E^-}^e$, but fixed.

\section{Nonconforming Trefftz virtual element methods} \label{section definition local space}
In this section, we construct a \ncTVEM\,for the approximation of solutions to the Helmholtz boundary value problem~\eqref{HH problem weak formulation}
and derive \textit{a priori} error estimates.

Our aim is to design a numerical method having the following structure:
\begin{equation} \label{HH problem weak discrete formulation}
\begin{cases}
\text{find } \un \in \Vnp \text{ such that}\\
\bn(\un,\vn) = \fn(\vn) \quad  \forall \vn \in  \Vnp,\\
\end{cases}
\end{equation} 
where, for all $n\in \mathbb N$, $\Vnp$ is a finite dimensional space subordinated to the mesh $\taun$, $\bn(\cdot,\cdot)~: \Vnp \times \Vnp \to \C$ is a {\it computable} sesquilinear form mimicking its continuous counterpart $b(\cdot,\cdot)$ defined in~\eqref{sesquilinearform b},
and the functional $\fn(\cdot):\, \Vnp \to \C$ is a {\it computable} counterpart of $\int_{\partial \Omega} \g \overline{v} \, \ds$.

The reason why we do not employ the continuous sesquilinear forms and right-hand side is that the functions in the \ncTVE\,spaces are not known in closed form;
therefore, the continuous sesquilinear forms and right-hand side are not computable.

The outline of this section is the following.
In Section \ref{subsection Trefftz-VE spaces}, local Trefftz-VE spaces, as well as the global space $\Vnp$ in~\eqref{HH problem weak discrete formulation}, are introduced.
Then, in Section \ref{subsection local projectors}, local projectors mapping from the local Trefftz-VE spaces into spaces of plane waves are defined;
such projectors will allow to define suitable $\bn(\cdot,\cdot)$ and $\fn(\cdot)$ in~\eqref{HH problem weak discrete formulation}, see Section \ref{subsection construction of the method}.
We point out that we do not discuss here the details of the implementation, but we refer the reader to~\cite{TVEM_Helmholtz_num}.

\subsection{Nonconforming Trefftz virtual element spaces} \label{subsection Trefftz-VE spaces}
The aim of this section is to specify the \ncTVE\,space $\Vnp$ that defines the method in~\eqref{HH problem weak discrete formulation}.

Given $p \in \N$, for every $n \in \N$ and $\E \in \taun$, we first introduce the {\it local} Trefftz-VE space
\begin{equation} \label{local Trefftz-VE space}
\VE:=\left\{\vn \in H^1(\E) \,\, | \,\, \Delta \vn+k^2 \vn = 0 \, \text{ in }\E, \quad \gammaIK(\vn)_{|_\e} \in \PWc(\e) \; \forall e \in \EE \right\},
\end{equation}
where we recall that $\gammaIK$ is defined in~\eqref{def impedance trace} and $\PWc(\e)$ is given in~\eqref{edge space with constant}.
In words, this space consists of all functions in $H^1(\E)$ in the kernel of the Helmholtz operator, whose impedance traces are edgewise equal to traces of plane waves including constants. 

It can be easily seen that $\PWE \subset \VE$.
However, the space $\VE$ also contains other functions not available in closed form, whence the term \textit{virtual} in the name of the method.

For each~$\E\in \taun$, the dimension~$\pE$ of the discrete space~$\VE$ in~\eqref{local Trefftz-VE space} coincides with the sum over all~$\e \in \EE$ of the dimension~$\pe$ of the edge plane wave spaces~$\PWc(\e)$ in~\eqref{edge space with constant}.
Therefore, we have
\begin{equation*}
\pE:=\dim \VE=\sum_{e \in \EE} \pe,
\end{equation*}
where we recall that $\pe \le \p+1$.

Having this, we define the set of local degrees of freedom on $\E \in \taun$ as the moments on each edge $\e \in \EE$ with respect to functions in the space $\PWc(\e)$ defined in~\eqref{edge space with constant}. More precisely, given $\vn \in \VE$, its degrees of freedom are
\begin{equation} \label{dofs}
\dof_{\e,\ell}(\vn) = \frac{1}{\he} \int_\e \vn \overline{\wle} \, \ds \quad  \forall e \in \EE, \, \forall \ell \in \Je.
\end{equation}
Besides, we denote by $\{ \varphi_{\e,\ell}  \}_{\e\in \EE,\, \ell \in \Je}$ the local canonical basis, where
\begin{equation} \label{local basis}
\dof_{\widetilde \e, \widetilde\ell} (\varphi_{\e,\ell}) = \delta_{(\e,\ell),(\widetilde \e, \widetilde \ell)} = \begin{cases}
1 & \text{if } (\e,\ell) = (\widetilde \e, \widetilde \ell) \\
0 & \text{otherwise}.\\
\end{cases}
\end{equation}

We prove that the set of degrees of freedom~\eqref{dofs} is unisolvent for all $\E \in \taun$, provided that the following assumption on the wave number $\k$ is satisfied:  
\begin{itemize}
\item[(\textbf{A1})] the wave number $\k$ is such that $\k^2$ is not a Dirichlet-Laplace eigenvalue on $\E$ for all $\E\in \taun$.
\end{itemize}
For a given $\E \in \taun$, the assumption (\textbf{A1}) results in a condition on the product $\hE  \k$. To be more precise, 
for any simply connected element $\E$, the smallest Dirichlet-Laplace eigenvalue on $\E$ satisfies
\[
\lambda_1\ge\frac{a}{\rho_\E^2},
\]
where $\rho_\E$ denotes the radius of the largest ball contained in $\E$ and where $a\ge 0.6197$, see e.g. \cite{banuelos1994brownian}.

As a consequence, assuming that
\begin{equation} \label{condition for Dirichlet-Laplace}
\hE \k \le \sqrt{c_0 a}
\end{equation}
for some $c_0\in (0,1]$, we deduce
\[
\k^2 = \frac{\hE^2 \k^2}{\hE^2} \le \frac{c_0 a}{\hE^2} \le \frac{c_0 a}{\rho_\E^2}\le \lambda_1,
\]
which means that the condition~\eqref{condition for Dirichlet-Laplace} guarantees that $\k^2$ is not a Dirichlet-Laplace eigenvalue on~$\E$.

\begin{lem} \label{lemma unisolvency}
Suppose that the assumption (\textbf{A1}) holds true. Then, for every $\E \in \taun$, the set of degrees of freedom~\eqref{dofs} is unisolvent for $\VE$.
\end{lem}
\begin{proof}
Given $\E \in \taun$, we first observe that the dimension of the local space $\VE$ is equal to the number of functionals in~\eqref{dofs}. Thus, we only need to prove that, given any $\vn \in \VE$ such that all degrees of freedom~\eqref{dofs} are zero, then $\vn=0$.
 
 To this end, we observe that an integration by parts, together with the fact that $\vn$ belongs to the kernel of the Helmholtz operator, yields
\begin{equation} \label{equation for unisolvency}
\begin{split}
& \vert \vn \vert^2_{1,\E} - \k^2 \Vert \vn \Vert^2_{0,\E} - \im\k\Vert \vn \Vert^2_{0,\partial \E} = \int_\E \nabla \vn \cdot \overline{\nabla \vn} \, \dx - \k^2 \int_\E \vn \overline{\vn} \, \dx - \im\k \int_{\partial \E} \vn \overline{\vn} \, \ds \\
& = \int_\E \vn \overline{\underbrace{(-\Delta \vn - \k^2 \vn)}_{=0}} \, \dx + \int_{\partial \E} \vn \overline{(\nabla \vn \cdot \nE + \im\k \vn)} \, \ds = \sum_{\e \in \mathcal E^\E} \int_{\e} \vn \overline{\gammaIK(\vn)_{|_e}} \, \ds = 0,\\
\end{split}
\end{equation}
where in the last identity we also used the facts that, owing to the definition of the space $\VE$ in~\eqref{local Trefftz-VE space}, the impedance trace of $\vn$ is an element of the space~\eqref{boundary space with constants},
and that the degrees of freedom~\eqref{dofs} of $\vn$ are zero. Thus, the imaginary part on the left-hand side of~\eqref{equation for unisolvency} is zero and one deduces that $\vn=0$ on $\partial \E$.
Since $\vn$ is also solution to a homogeneous Helmholtz equation, the assertion follows thanks to the assumption (\textbf{A1}).
\end{proof}
The \emph{global} Trefftz-VE space is given by
\begin{equation} \label{global Trefftz space}
\Vnp:=\{ \vn \in H^{1,nc}(\taun): \, v_{\h|_\E} \in \VE \quad \forall \E \in \taun \},
\end{equation}
where $H^{1,nc} (\taun)$ is defined in~\eqref{global non conforming space} with the same uniform~$\p$ as in the definition of the local Trefftz-VE spaces in~\eqref{local Trefftz-VE space}.
Consequently, the set of global degrees of freedom is obtained by coupling the local degrees of freedom on the interfaces between elements.

We underline that the definition of the degrees of freedom~\eqref{dofs} is actually tailored for building discrete trial and test spaces that are nonconforming in the sense of~\eqref{global non conforming space}.
Besides, they will be used in the construction of projectors mapping onto spaces of plane waves. This is the topic of the next Section \ref{subsection local projectors}.

\begin{remark} \label{remark actual dimension of the space}
Under the choice of the degrees of freedom in~\eqref{dofs}, the dimension of the global space is larger than that of plane wave discontinuous Galerkin methods \cite{MoiolaPhDthesis}.
However, at the practical level,
an ed\-ge-by-ed\-ge or\-tho\-go\-na\-li\-za\-tion-and-fil\-te\-ring process can be implemented, in order to reduce the dimension of the \ncTVE\,space without losing in terms of accuracy.
This procedure is described in detail in~\cite{TVEM_Helmholtz_num}. There, it is also observed that the two methods have a comparable behavior in terms of accuracy versus number of degrees of freedom and,
in some occasions, the method presented herein performs better.
\end{remark}

\subsection{Local projectors} \label{subsection local projectors}
In this section, we introduce local projectors mapping functions in local Trefftz-VE spaces~\eqref{local Trefftz-VE space} onto plane waves.
Such projectors will play a central role in the construction of the computable sesquilinear form $\bn(\cdot,\cdot)$ and functional $\fn(\cdot)$ for the method~\eqref{HH problem weak discrete formulation}.

To start with, given $\E \in \taun$, we define the local projector
\begin{equation} \label{projector aK}
\begin{split}
\Pip: \, &\VE \rightarrow \PWE \\
&\aE(\Pip \un, \wE) = \aE(\un, \wE) \quad \forall \un \in \VE,\, \forall \wE \in \PWE.
\end{split}
\end{equation}
Note that this projector is {\it computable} by means of the degrees of freedom~\eqref{dofs} without the need of explicit knowledge of the functions of $\VE$.
Indeed, an integration by parts and the fact that any plane wave $\wE\in \PWE$ belongs to the kernel of the Helmholtz operator lead to
\begin{equation*}
\begin{split}
\aE(\un ,\wE) 	& = \int_{\E} \nabla \un \cdot \overline{\nabla \wE} \, \dx -k^2 \int_{\E} \un \overline{\wE} \, \dx \\
& = \sum_{\e \in \EE} \int_{\e} \un \overline{(\nabla \wE \cdot \nE)} \, \ds \quad \forall \un \in \VE,\, \forall \wE \in \PWE.\\
\end{split}
\end{equation*}
Since $(\nabla \wE \cdot \n)_{\E|_e} \in \PWc(e)$ for all $e \in \EE$, computability is guaranteed by the choice of the degrees of freedom in~\eqref{dofs}. 

In the following proposition we prove that $\Pip$ is well-defined.
\begin{prop}
Assume that $\E$ is an element of a mesh that satisfies the assumption (\textbf{G1}). Then, the following two statements hold true:
\begin{enumerate}
\item Denoting by $\mu_2$ the smallest positive Neumann-Laplace eigenvalue in $\E$, it holds 
\begin{equation*}
\mu_2 \geqslant \frac{\CDelta \pi^2}{h^2_\E},
\end{equation*}
where $\CDelta \in (0,1]$ only depends on the shape of $\E$, i.e., on $\rho_0$ and $\rho$ in the assumption (\textbf{G1}).
\item Assume that the assumption (\textbf{D1}) holds true. If $\hE k$ is such that there exists a constant $C_1>0$ with
\begin{equation*}
0<\hE k < C_1 \le \min\left\{\frac{\sqrt{\CDelta} \pi}{\sqrt{2}}, 0.5538 \right\},
\end{equation*}
then $\k^2<\mu_2$, and in particular it follows that $\Pip$ is well-defined and continuous. More precisely, there exists a constant $\beta(\hE k)>0$, uniformly bounded away from zero as $\hE k \to 0$, such that 
\begin{equation*}
\lVert \Pip \un \rVert_{1,\k,\E} \le \frac{1}{\beta(\hE \k)} \lVert \un \rVert_{1,\k,\E} \quad \forall \un \in \VE.
\end{equation*}
\end{enumerate}
Note that, whenever $\E$ is convex, $\CDelta=1$, see e.g.~\cite{payne1960optimal}, and hence $\min\left\{\frac{\sqrt{\CDelta} \pi}{\sqrt{2}}, 0.5538 \right\} = 0.5538.$
\end{prop}
\begin{proof}
For the proof of the first part, one can refer to \cite{bramble1962bounds,ladeveze1978bounds}, and for the second part, to \cite[Propositions 2.1 and 2.3]{Helmholtz-VEM}.
\end{proof}
Given a function~$\vn \in \VE$, in addition to the projector~$\Pip$, we define on every edge~$\e \in \EE$ the projector
\begin{equation} \label{projector edge L2}
\begin{split}
\Pie: \, & \VE_{|_\e} \rightarrow \PWc(\e) \\
&\int_\e \Pie (\vn{}_{|\e}) \overline{\we} \, \ds = \int_e \vn{}_{|_\e} \overline{\we} \, \ds \quad \forall \vn\in\VE,\,\forall \we \in \PWc(\e).
\end{split}
\end{equation}
The computability of this projector for functions in $\VE$ is again provided by the choice of the degrees of freedom in~\eqref{dofs}.
Clearly, $\Pie (\un{}_{|\e})$ coincides with the $L^2(\e)$ projection of~$\un{}_{|_\e}$ onto~$\PWc(\e)$.
\begin{remark}
The projector $\Pie$ is not defined for functions in the nonconforming space $\Vnp$ in~\eqref{global Trefftz space},
but rather for the restrictions of such functions to the elements of the mesh.
However, in order to avoid a cumbersome notation in the following, we will not highlight such restrictions whenever it is clear from the context.
\end{remark}

The following approximation result holds true.
\begin{prop} \label{proposition best approx edge}
Let $\E \in \taun$ and $e \in \EE$. For all $u \in H^1(\E)$, it holds
\begin{equation} \label{bound u-Piu}
\Vert u - \Pie u \Vert_{0,\e} \le \CTP \hE^{\frac{1}{2}} \vert u - \wE \vert_{1,\E} \quad \forall \wE \in \PWE,
\end{equation}
where the constant $C_0>0$ only depends on the shape of $\E$.
\end{prop}
\begin{proof}
We first note that, for each $\E \in \taun$ and $e \in \EE$, the definition of $\Pie$ in~\eqref{projector edge L2} yields
\[
\Vert u - \Pie u \Vert_{0,\e} \le \Vert u - c - \wE \Vert_{0,\e}\le \Vert u - c - \wE \Vert_{0,\partial \E} \quad \forall \wE \in \PWE, \, \forall c \in \mathbb C.
\]
By selecting
\begin{equation*}
c=\frac{1}{|\E|} \int_{\E} (u - \wE) \, \dx,
\end{equation*}
and by using the trace inequality~\eqref{trace inequality} together with the Poincar\'e-Friedrichs inequality~\eqref{Poincare Friedrich full}, we get
\begin{equation*} 
\begin{split}
\lVert u - c - \wE \rVert^2_{0,\partial \E} 
&\le \CT \left(\hE^{-1} \lVert u - c - \wE \rVert^2_{0,\E} + \hE \lvert u  - \wE \rvert^2_{1,\E} \right) \\
&\le \CT (\CP^2+1) \hE \lvert u - \wE \rvert^2_{1,\E},
\end{split}
\end{equation*}
from which we have~\eqref{bound u-Piu} with $\CTP^2:=\CT (\CP^2+1)$.
\end{proof}
For future use, we also denote by $\Pib$ the $L^2$ projector
\begin{equation} \label{projector boundary L2}
\begin{split}
\Pib: \, L^2(\partial \Omega) \rightarrow \prod_{e \in \Enb} \PWc(\e).
\end{split}
\end{equation}
We highlight that, for any~$\vn \in V_\h$, the identity~$(\Pib(\vn{}_{|_{\partial \Omega}}))_{|_\e} = \Pie (\vn{}_{|_\e})$ holds for all boundary edges~$\e \in \Enb$.

\subsection{Discrete sesquilinear forms and right-hand side} \label{subsection construction of the method}
We introduce the sesquilinear form $\bn(\cdot,\cdot)$ and the functional $\fn(\cdot)$ characterizing the method~\eqref{HH problem weak discrete formulation}. 

\subsubsection*{Construction of $\bn(\cdot,\cdot)$}
Following the VEM gospel \cite{VEMvolley}, the definition of $\Pip$ in~\eqref{projector aK} gives
\begin{equation} \label{Pythagoras theorem}
\aE(\un,\vn) = \aE( \Pip \un, \Pip \vn) + \aE( (I-\Pip) \un, (I-\Pip) \vn) \quad \forall \un,\,\vn \in \VE.
\end{equation}
The first term on the right-hand side of~\eqref{Pythagoras theorem} is computable, whereas the second one is not, and thus has to be replaced by a proper \emph{computable} sesquilinear form $\SE(\cdot,\cdot)$, which henceforth goes under the name of \textit{stabilization};
see~\eqref{D-recipe} for an explicit choice. Having this, we set
\begin{equation} \label{local discrete bf}
\anE(\un,\vn) := \aE( \Pip \un, \Pip \vn) + \SE\left( (I-\Pip) \un, (I-\Pip) \vn \right) \quad \forall \un,\,\vn \in \VE.
\end{equation}
We point out that the local sesquilinear form $\anE(\cdot,\cdot)$ satisfies the following \emph{plane wave consistency property}:
\begin{equation} \label{pw consistency}
\anE(\wE,\vn) = \aE(\wE,\vn), \quad \anE(\vn,\wE) = \aE(\vn,\wE) \quad \forall \wE \in \PWE, \, \forall \vn \in \VE.
\end{equation}
Moreover, we replace the boundary integral term in $\b(\cdot,\cdot)$ in~\eqref{sesquilinearform b} with
\begin{equation*}
\im\k \int_{\partial \Omega} \un \overline{\vn} \, \ds 
\quad \mapsto \quad \im\k \int_{\partial \Omega} (\Pib \un) \overline{(\Pib \vn)} \, \ds \quad \forall \un,\vn \in \Vnp,
\end{equation*}
where $\Pib$ is defined in~\eqref{projector boundary L2}. Altogether, the global sesquilinear form $\bn(\cdot,\cdot)$ in~\eqref{HH problem weak discrete formulation} is given by
\begin{equation} \label{definition bn}
\bn(\un,\vn):=\an(\un,\vn) + \im\k \int_{\partial \Omega} (\Pib \un) \overline{(\Pib \vn)} \, \ds  \quad \forall \un,\vn \in \Vnp,
\end{equation} 
where 
\begin{equation*}
\an(\un,\vn):=\sum_{\E \in \taun} \anE(\un,\vn).
\end{equation*}
\medskip

In the subsequent error analysis of the method~\eqref{HH problem weak discrete formulation}, see Theorem~\ref{theorem abstract theory} below,
we will require continuity of the local sesquilinear forms $\anE(\cdot, \cdot)$ given in~\eqref{local discrete bf}, as well as a discrete G\aa{}rding inequality for $\bn(\cdot, \cdot)$ defined in~\eqref{definition bn}.

If the stabilization forms~$\SE(\cdot,\cdot)$ satisfy
\begin{equation} \label{equations 34 and 35}
\alphah \Vert \vn \Vert_{1,\k,\E}^2 - 2\k^2 \Vert \vn \Vert^2_{0,\E} \le \SE(\vn,\vn) \le \gamma_\h \Vert \vn \Vert^2_{1,\k,\E}\quad \forall \vn \in \ker (\Pip),\, \forall \E \in \taun,
\end{equation}
for some positive constants~$\alphah$ and~$\gamma_\h$, then, by proceeding as in Theorem \cite[Proposition 4.1]{Helmholtz-VEM},
one can prove that the local continuity assumptions and the local G\r arding inequalities of Theorem~\ref{theorem abstract theory} are satisfied.

\subsubsection*{Construction of $\fn(\cdot)$}
We set $\gn:=\Pib g$, where $\Pib$ is defined in~\eqref{projector boundary L2}.
Using the approximation $\gn$ instead of $\g$, allows us to define the computable functional
\begin{equation} \label{discrete rhs}
\fn(\vn):=\int_{\partial \Omega} \gn \overline{\vn} \, \ds
= \int_{\partial \Omega} \g \overline{(\Pib \vn)} \, \ds.
\end{equation}
In order to avoid additional complications in the forthcoming analysis, we will assume that the integral in~\eqref{discrete rhs} can be computed exactly. In practice, one approximates such integrals using high order quadrature formulas.

\section{\textit{A priori} error analysis} \label{section a priori error analysis}
In this section, we first prove approximation properties of functions in Trefftz-VE spaces in Section~\ref{subsection best approximation estimates in Trefftz-VE spaces}.
Then, in Section~\ref{subsection abstract error analysis}, we deduce an abstract error result which is instrumental for the derivation of \textit{a priori} error estimates in Section \ref{subsection a priori error bound}.



\subsection{Approximation properties of functions in Trefftz virtual element spaces} \label{subsection best approximation estimates in Trefftz-VE spaces}
In order to discuss the approximation properties for the \ncTVE\,spaces,
we recall the following local $\h$-version best approximation result from \cite[Theorem 5.2]{moiola2009approximation} for plane wave spaces in two dimensions.

\begin{thm} \label{theorem local h-best approximation plane waves}
Assume that $\E$ is an element of a mesh $\taun$ satisfying the assumption (\textbf{G1}). In addition, let $u \in H^{s+1}(\E)$, $s \in \mathbb R_{\ge 1}$, be such that $\Delta u +\k^2 u=0$,
and let $\PWE$ be the plane wave space with directions $\{ \dl \}_{\ell=1,\dots,p}$, $p = 2q+1$, $q \in \N_{\ge 2}$, satisfying the assumption (\textbf{D1}).
Then, for every $L \in \mathbb R$ with $1 \le L \le \min(q,s)$, there exists $\wE \in \PW_p(\E)$ such that, for every $0 \le j \le L$, it holds
\begin{equation*}
\lVert u-\wE \rVert_{j,\k,\E} \le \cPW (\k\hE) \hE^{L+1-j} \lVert u \rVert_{L+1,\k,\E},
\end{equation*}
where 
\begin{equation} \label{cPW definition}
\cPW (t) := C e^{\left( \frac{7}{4}-\frac{3}{4}\rho \right)t} \left( 1+t^{j+q+8} \right),
\end{equation}
and the constant $C>0$ depends on $q$, $j$, $L$, $\rho$, $\rho_0$, and the directions $\{ \dl \}$, but is independent of $\k$, $\hE$, and $u$.
Note that the constant $\cPW(\k\hE)$ in~\eqref{cPW definition} is uniformly bounded as $\hE \to 0$.
\end{thm}

In the ensuing result, we prove that the best approximation error of functions in the \ncTVE\,space $\Vnp$ can be bounded by the best error in (discontinuous) plane wave spaces.

\begin{thm} \label{theorem best approximation result}
Consider a family of meshes $\{\taun\}_{n \in \N}$ satisfying the assumptions (\textbf{G1})-(\textbf{G3}) and (\textbf{A1}), and let $\Vnp$ be the nonconforming Trefftz-VE space defined in~\eqref{global Trefftz space}
with directions $\{ \dl \}_{\ell=1,\dots,p}$, $p = 2q+1$, $q \in \N_{\ge 2}$, satisfying the assumption (\textbf{D1}). Further, assume that, on every element $\E \in \taun$, $k$ and $\hE$ are such that $\k \hE$ is sufficiently small,
see condition~\eqref{best approx condition on h,k} below. Then, for any $u \in H^1(\Omega)$, there exists a function $\uI \in \Vnp$ such that
\begin{equation} \label{best approximation u-uI}
\Vert u - \uI \Vert_{1,\k,\taun} \le \cBA(\k\h) \lVert u - \uPW \rVert_{1,\k,\taun} \quad \forall \uPW \in \PWtaun,
\end{equation} 
where
\begin{equation} \label{def cBA}
\cBA(t):=2 \frac{\delta}{\delta-1} (1+\CP^2 t^2) \left(\CT (\CP^2+1) t + 2\right),
\end{equation}
with $\CT$ from~\eqref{trace inequality}, $\CP$ from~\eqref{Poincare Friedrich partial}, and $\delta>1$ from condition~\eqref{best approx condition on h,k} below, remains uniformly bounded as $t \to 0$.
\end{thm}
\begin{proof}
Given $u \in H^1(\Omega)$, we define its ``interpolant'' $\uI$ in $\Vnp$ in terms of its degrees of freedom as follows:
\begin{equation} \label{definition TVEM interpolant}
\int_\e (\uI - u) \overline{\wle} \, \ds = 0 \quad \forall \ell \in \Je, \, \forall \e \in \EE, \, \forall \E \in \taun,
\end{equation}
where the functions $\wle$ are defined in~\eqref{plane wave edge}.

We stress that, with this definition, $\uI$ is automatically an element of~$\Hnc$ introduced in~\eqref{global non conforming space}.
Moreover, the definition~\eqref{definition TVEM interpolant} implies that the average of $u-\uI$ on every edge $e \in \EE$, $\E \in \taun$, is zero, thanks to the fact that
the space $\PWc(\e)$ contains the constants for all edges $\e$.
This, together with the Poincar\'{e}-Friedrichs inequality~\eqref{Poincare Friedrich partial}, gives, for each element $\E \in \taun$,
\begin{equation} \label{Poincare inequality}
\Vert u - \uI \Vert_{0,\E} \le \CP \hE \vert u -\uI \vert_{1,\E}.
\end{equation}
In order to obtain~\eqref{best approximation u-uI}, we start by proving local approximation estimates. To this end, let  $\E \in \taun$ be fixed. By using the triangle inequality, we obtain
\begin{equation} \label{initial triangular inequality}
\vert u - \uI \vert _{1,\E} \le \vert u - \wE \vert_{1,\E} + \vert \uI - \wE \vert_{1,\E} \quad \forall \wE \in \PWE.
\end{equation}
For what concerns the second term, by using an integration by parts, taking into account that both $\uI$ and $\wE$ belong to the kernel of the Helmholtz operator,
and by employing the definition of the impedance trace $\gammaIK$, we get, for every constant $\cE \in \mathbb C$,
\begin{equation} \label{best approx relation 1}
\begin{split}
\vert \uI - \wE \vert^2_{1,\E} &= \int_{\E} \nabla (\uI - \wE) \cdot \overline{\nabla(\uI - \wE)} \, \dx  \\
& = -\int_\E \Delta (\uI - \wE)\overline{(\uI - \wE)} \, \dx+ \int_{\partial \E} \nabla(\uI - \wE -\cE) \cdot \nE \, \overline{(\uI - \wE)} \, \ds  \\
& = \k^2 \int_\E (\uI - \wE) \overline{(\uI - \wE)} \, \dx +\int_{\partial \E} \gammaIK(\uI-\wE -\cE) \overline{(\uI-\wE)} \, \ds  \\
&\quad - \im\k \int_{\partial \E} (\uI-\wE -\cE)\overline{(\uI-\wE)} \, \ds.
\end{split}
\end{equation}
Taking now into account that $\gammaIK(\uI-\wE -\cE)_{|_e}$ belongs to the space $\PWc(\e)$ introduced in~\eqref{edge space with constant}, for each edge $e \in \EE$, 
the definition of $\uI$ in~\eqref{definition TVEM interpolant} implies
\begin{equation} \label{relation interpolant}
\int_{\partial \E} \gammaIK(\uI-\wE -\cE) \overline{(\uI-\wE)} \, \ds = \int_{\partial \E} \gammaIK(\uI-\wE -\cE) \overline{(u-\wE)} \, \ds.
\end{equation}
Using the definition of impedance traces, inserting~\eqref{relation interpolant} in~\eqref{best approx relation 1}, integrating by parts back, and using that both $\uI$ and $\wE$ belong to the kernel of the Helmholtz operator lead to
\begin{equation} \label{best approx relation 2} 
\begin{split}
\vert \uI - \wE \vert^2_{1,\E} 
& = \k^2 \int_\E (\uI - \wE) \overline{(\uI - \wE)} \, \dx +\int_{\E} \nabla(\uI-\wE) \cdot \overline{\nabla(u-\wE)} \, \dx  \\
&\quad + \int_{\E} \Delta(\uI-\wE) \overline{(u-\wE)} \, \dx + \im\k \int_{\partial \E} (\uI-\wE -\cE) \overline{(u-\uI)} \, \ds  \\
& = \k^2 \int_\E (\uI - \wE) \overline{(\uI - u)} \, \dx + \int_{\E} \nabla(\uI-\wE) \cdot \overline{\nabla(u-\wE)} \, \dx  \\ 
&\quad + \im\k \int_{\partial \E} (\uI-\wE -\cE) \overline{(u-\uI)} \, \ds =: Z_1 + Z_2 + Z_3.
\end{split}
\end{equation}
We bound the three terms on the right-hand side of~\eqref{best approx relation 2} separately.
For $Z_1$, we use the Cauchy-Schwarz and the triangle inequalities, the inequality~\eqref{Poincare inequality}, and the bound $a^2+ab \leqslant \frac{1}{2} (3a^2+b^2)$, to get
\begin{equation} \label{bound first term}
\begin{split}
\vert Z_1\vert &= \left \vert  \k^2 \int_\E (\uI- \wE) \overline{(\uI-u)} \, \dx \right \vert 
\le \k^2 \left( \Vert u - \uI \Vert^2_{0,\E} + \Vert u - \wE\Vert_{0,\E} \Vert u - \uI \Vert_{0,\E} \right) \\
& \le \k^2 \left\{ \CP^2 \hE^2 \vert u - \uI \vert^2_{1,\E} + \CP \hE \vert u - \uI \vert_{1,\E} \Vert u -\wE \Vert_{0,\E} \right\} \\
& \le \frac{\k^2}{2} \left\{ 3\CP^2 \hE^2\vert u - \uI \vert^2_{1,\E} +  \Vert u- \wE\Vert^2_{0,\E} \right\}.
\end{split}
\end{equation}
The term $Z_2$ can be bounded by applying the Cauchy-Schwarz inequality and the bound $ab \le \frac{1}{2}(a^2+b^2)$:
\begin{equation} \label{bound second term}
\begin{split}
\vert Z_2 \vert = \left\vert \int_\E \nabla(\uI - \wE) \cdot \overline{\nabla(u-\wE)} \, \dx \right\vert 
\le \frac{1}{2} \left(\vert \uI -\wE \vert^2_{1,\E} + \vert u - \wE \vert^2_{1,\E}\right).
\end{split}
\end{equation}
Finally, for the term $Z_3$, by employing the Cauchy-Schwarz and the triangle inequalities, and again the bound $a^2+ab \leqslant \frac{1}{2} (3a^2+b^2)$, we obtain
\begin{equation} \label{bound third term}
\begin{split}
\vert Z_3 \vert & = \left\vert \im\k \int_{\partial \E} (\uI - \wE -\cE) \overline{(u - \uI)} \, \ds \right\vert  \\
& \le \k \left( \Vert u - \uI \Vert^2_{0,\partial \E} + \Vert u - \wE  -\cE\Vert_{0,\partial \E} \Vert u - \uI \Vert_{0,\partial \E} \right)\\
& \le \frac{\k}{2} \left( 3\Vert u - \uI \Vert^2_{0,\partial \E} + \Vert u - \wE  -\cE\Vert^2_{0,\partial \E} \right).
\end{split}
\end{equation}
Combining the trace inequality~\eqref{trace inequality} with~\eqref{Poincare inequality} yields
\begin{equation} \label{best approx relation 3}
\Vert u - \uI \Vert^2_{0,\partial \E}  \le  \CT \left(\hE^{-1} \lVert u - \uI \rVert^2_{0,\E} + \hE \lvert u - \uI \rvert^2_{1,\E} \right)
\le \CT (\CP^2+1) \hE \vert u - \uI \vert^2_{1,\E}.
\end{equation}
Similarly, making use of the trace inequality~\eqref{trace inequality} and the Poincar\'{e}-Friedrichs inequality~\eqref{Poincare Friedrich full}, after selecting $\cE=\frac{1}{|\E|} \int_\E (u-\wE) \, \dx$, leads to 
\begin{equation} \label{best approx relation 4}
\begin{split}
\Vert u - \wE -\cE \Vert^2_{0,\partial \E}  &\le \CT \left(\hE^{-1} \lVert u - \wE  -\cE\rVert^2_{0,\E} + \hE \lvert u - \wE \rvert^2_{1,\E} \right) \\
&\le \CT (\CP^2+1) \hE \lvert u - \wE \rvert^2_{1,\E}.
\end{split}
\end{equation}
By plugging~\eqref{best approx relation 3} and~\eqref{best approx relation 4} into~\eqref{bound third term}, we obtain
\begin{equation} \label{bound third term part 2}
\vert Z_3 \vert \le 
\frac{1}{2} \CT (\CP^2+1) \k\hE \left( 3\vert u - \uI \vert^2_{1,\E} + \vert u - \wE \vert^2_{1,\E} \right).
\end{equation}
Inserting the three bounds~\eqref{bound first term},~\eqref{bound second term}, and~\eqref{bound third term part 2} into~\eqref{best approx relation 2}, and moving the contribution $\frac{1}{2} \vert \uI -\wE \vert^2_{1,\E}$ to the left-hand side, yield
\begin{equation} \label{best approx relation 5}
\begin{split}
\frac{1}{2} \vert \uI - \wE \vert^2_{1,\E} &\le
\frac{3}{2} \k\hE \left( \CP^2 \k\hE +  \CT (\CP^2+1) \right) \vert u - \uI \vert^2_{1,\E} \\
&\quad + \frac{\k^2}{2} \Vert u- \wE\Vert^2_{0,\E}  
+ \frac{1}{2} \left(1+ \CT (\CP^2+1) \k\hE \right)\vert u - \wE \vert^2_{1,\E}.
\end{split}
\end{equation}
From~\eqref{initial triangular inequality}, the bound $(a+b)^2 \leqslant 2(a^2+b^2)$, and~\eqref{best approx relation 5}, we get, further taking the definition of the norm $\lVert \cdot \rVert_{1,\k,\E}$ into account,
\begin{equation*} 
\begin{split}
\vert u - \uI \vert^2 _{1,\E} 
&\le 2 \vert u - \wE \vert^2_{1,\E} + 2 \vert \uI - \wE \vert^2_{1,\E} \\
&\le 6\k\hE (\CP^2 \k\hE+\CT (\CP^2+1))  \vert u - \uI \vert^2_{1,\E} + 2\k^2 \lVert u - \wE  \rVert^2_{0,\E} \\
&\quad + 2\left(\CT (\CP^2+1) \k\hE + 2\right) \lvert u - \wE \rvert^2_{1,\E} \\
& \le 6\k\hE (\CP^2 \k\hE+\CT (\CP^2+1))  \vert u - \uI \vert^2_{1,\E} + 2\left(\CT (\CP^2+1) \k\hE + 2\right) \lVert u - \wE \rVert^2_{1,\k,\E}.
\end{split}
\end{equation*}
Under the assumption that $\k$ and $\hE$ are such that
\begin{equation} \label{best approx condition on h,k}
6\k\hE (\CP^2 \k\hE+\CT (\CP^2+1)) \le \frac{1}{\delta}
\end{equation}
for some $\delta>1$, we obtain
\begin{equation} \label{local estimate u-uI in H1 seminorm}
\begin{split}
\vert u - \uI \vert^2 _{1,\E} 
&\le 2\frac{\delta}{\delta-1} \left(\CT (\CP^2+1) \k\hE + 2\right) \lVert u - \wE \rVert^2_{1,\k,\E} \quad \forall \wE \in \PWE.
\end{split}
\end{equation}
From the definition of the norm $\lVert \cdot \rVert_{1,\k,\E}$ in~\eqref{definition local weighted Sobolev norm}, inequality~\eqref{Poincare inequality}, and the estimate~\eqref{local estimate u-uI in H1 seminorm}, we get
\begin{equation*}
\begin{split}
\lVert &u-\uI \rVert_{1,\k,\E}^2
=\lvert u-\uI \rvert_{1,\E}^2 + \k^2 \lVert u-\uI \rVert_{0,\E}^2
\le (1+\CP^2 (\k \hE)^2) \vert u -\uI \vert^2_{1,\E} \\
&\le 2\frac{\delta}{\delta-1} (1+\CP^2 (\k \hE)^2) \left(\CT (\CP^2+1) \k\hE + 2\right) \lVert u - \wE \rVert^2_{1,\k,\E}.
\end{split}
\end{equation*}
The assertion follows by summing over all elements $\E \in \taun$ and taking the square root.
\end{proof}
By combining Theorem \ref{theorem local h-best approximation plane waves} with Theorem \ref{theorem best approximation result}, we have the following best approximation error bound.
\begin{cor} \label{corollary balancing errors}
Under the assumptions of Theorems \ref{theorem local h-best approximation plane waves} and \ref{theorem best approximation result},
for $u \in H^{s+1}(\Omega)$, $s \in \R_{\ge 1}$, satisfying $\Delta u +\k^2 u=0$,  the following holds true:
\begin{equation*} 
\begin{split}
\Vert u - \uI \Vert_{1,\k,\taun} 
\le C^\ast(\k\h) \h^{\zeta} \lVert u \rVert_{\zeta+1,\k,\taun},
\end{split}
\end{equation*}
where $\zeta:=\min(q,s)$ and
\begin{equation*}
C^\ast(t):=C e^{\left( \frac{7}{4}-\frac{3}{4}\rho \right)t} \left( 1+t^{q+9} \right) 2 \frac{\delta}{\delta-1} (1+\CP^2 t^2) \left(\CT (\CP^2+1) t + 2\right),
\end{equation*}
with $C>0$ depending on $q$, $\rho$, $\rho_0$, and $\{ \dl \}_{\ell=1,\dots,\p}$, but independent of $\k$, $\h$, and~$u$, and with $\delta$ being as in Theorem~\ref{theorem best approximation result}.
\end{cor}

\subsection{Abstract error analysis} \label{subsection abstract error analysis}
In this section, we prove existence and uniqueness of the discrete solution to the method~\eqref{HH problem weak discrete formulation},
and derive \textit{a priori} error bounds, provided that the mesh size is sufficiently small.

To this purpose, we consider a variational formulation of~\eqref{HH continuous problem} obtained by testing~\eqref{HH continuous problem} with functions in $\Vnp$. Given $u$ the exact solution to problem~\eqref{HH continuous problem}, we have, for all functions $\vn \in \Vnp$, 
\begin{equation*}
0=\sum_{K \in \taun} \int_{K} (-\Delta u -k^2 u) \overline{\vn} \, \dx
=\sum_{K \in \taun} \left[ \int_{K} \left( \nabla u \cdot \overline{\nabla \vn} -k^2 u \overline{\vn} \right) \, \dx - \int_{\partial K} (\nabla u \cdot \textbf{\textup{n}}_K) \overline{\vn} \, \ds \right]
\end{equation*}
and therefore
\begin{equation} \label{exact discrete method}
\sum_{K \in \taun} a^K(u,\vn) - \Nn(u,\vn) + \im\k \int_{\partial \Omega} u \overline{\vn} \, \ds = \int_{\partial \Omega} g \overline{\vn} \, \ds,
\end{equation}
where the nonconformity term $\Nn(\cdot,\cdot)$ is defined as
\begin{equation} \label{definition Nn}
\Nn(u,\vn):=\sum_{K \in \taun}  \int_{\partial K \backslash \partial \Omega} (\nabla u \cdot \textbf{\textup{n}}_K) \, \overline{\vn} \, \ds
= \sum_{e \in \EnI} \int_{e} \nabla u \cdot \overline{\llbracket \vn \rrbracket} \, \ds.
\end{equation}
We have now all the ingredients to prove the following abstract error result.

\begin{thm} \label{theorem abstract theory}
Let the assumptions (\textbf{G1})-(\textbf{G3}), (\textbf{D1}) and (\textbf{A1}) hold true; moreover, assume that the solution $u$ to~\eqref{HH problem weak formulation} belongs to $H^2(\Omega)$.
Further, let the number of plane waves be~$\p=2q+1$, $q\in \mathbb N_{\ge 2}$, and the local stabilization forms~$\SE(\cdot,\cdot)$ be such that the following properties are valid:
\begin{itemize}
\item \textup{(local discrete continuity)} there exists a constant~$\gamma_h>0$ such that
\begin{equation} \label{theorem ass continuity}
\vert \anE(\vn,\zn) \vert  \le \gamma_\h \lVert \vn \rVert_{1,\k,\E} \lVert \zn \rVert_{1,\k,\E} \quad \forall \vn, \zn \in \VE,\, \forall \E \in \taun;
\end{equation}
\item \textup{(discrete G\r arding inequality)} there exists a constant~$\alpha_h>0$ such that
\begin{equation} \label{theorem ass Garding}
\Re[\bn(\vn,\vn)] + 2\k^2 \lVert \vn \rVert_{0,\Omega}^2  \ge \alpha_h \lVert \vn \rVert^2_{1,\k,\taun} \quad \forall \vn \in \Vnp.
\end{equation}
\end{itemize}
Then, provided that~$\k$ and~$\h$ are chosen such that~$\k^2 \h$ is sufficiently small, see condition~\eqref{h small condition} below, the method~\eqref{HH problem weak discrete formulation} admits a unique solution~$\un \in \Vnp$ which satisfies
\begin{equation} \label{abstract error estimate}
\begin{split}
\Vert u - \un \Vert_{1,\k,\taun} 
&\lesssim \aleph_1(\k,\h) \Vert u - \uPW \Vert_{1,\k,\taun} 
+ \h \, \aleph_2(\k,\h) \vert u - \uPW \vert_{2,\taun} \\
& \quad + \h^{\frac{1}{2}} \, \aleph_2(\k,\h) \Vert \g - \Pib \g \Vert_{0,\partial \Omega} \quad \forall \uPW \in \PWtaun,
\end{split}
\end{equation}
with 
\begin{equation} \label{alephs}
\begin{split}
\aleph_1(\k,\h)&:=\left\{\frac{(\k\h+\gamma_h+1) (\k  (\varsigma(\k,\h)+\vartheta(\k,\h)) + \cBA(\k\h) +1)}{\alpha_h} +\cBA(\k\h)  \right\}, \\
\aleph_2(\k,\h)&:=\frac{\k (1 + \vartheta(\k,\h)) + 1}{\alpha_h},
\end{split}
\end{equation}
where $\Pib$ is defined in~\eqref{projector boundary L2}, the hidden constants in~\eqref{abstract error estimate} are independent of $\h$ and $\k$, and
\begin{equation} \label{definition beta}
\varsigma(\k,\h):=(1+\k\h)(1+\dOmega\k)\h, \quad 
\vartheta(\k,\h):=\cBA(\k\h) \varsigma(\k,\h),
\end{equation}
$\cBA(\k\h)$ being given in~\eqref{def cBA} and $\dOmega$ a positive constant depending only on $\Omega$.
\end{thm}
\begin{proof}
We prove the error bound~\eqref{abstract error estimate} under a condition on $\k^2 \h$ in five steps. Existence and uniqueness of discrete solutions, under the same assumption on $\k^2\h$, will follow as in \cite{schatz1974observation}.

\medskip
\noindent\textit{Step~1: Triangle inequality:}
Let $\un$ satisfy~\eqref{HH problem weak discrete formulation}. By the triangle inequality, we get
\begin{equation} \label{triangular inequality in abstract analysis}
\Vert u - \un \Vert_{1,\k,\taun} \le \Vert u - \uI \Vert_{1,\k,\taun} + \Vert \un-\uI \Vert_{1,\k,\taun},
\end{equation}
where $\uI \in \Vnp$ is defined as in~\eqref{definition TVEM interpolant}.
The first term on the right-hand side of~\eqref{triangular inequality in abstract analysis} can be bounded by using Theorem~\ref{theorem best approximation result}. We focus on the second one.
By setting~$\deltan:=\un-\uI \in \Vnp$ and using the discrete G\r arding inequality~\eqref{theorem ass Garding}, we obtain
\begin{equation} \label{applying Garding}
\alpha_h \Vert \deltan \Vert^2_{1,\k,\taun} \le \Re\left[ \bn(\deltan,\deltan) \right] + 2\k^2 \Vert \deltan\Vert_{0,\Omega}^2 =: I + II.
\end{equation}
\textit{Step~2: Estimate of the term~$I$ in~\eqref{applying Garding}:}
The identity in~\eqref{HH problem weak discrete formulation}, the definitions of $\bn(\cdot,\cdot)$ in~\eqref{definition bn}, of $F_\h(\cdot)$ in~\eqref{discrete rhs}, and of the projector $\Pib$ in~\eqref{projector boundary L2},
together with the plane wave consistency property~\eqref{pw consistency} yield
\[
\begin{split}
\bn(\deltan,\deltan) 	
&= \bn(\un-u_I,\deltan) = \bn(\un,\deltan) - \bn(u_I,\deltan) \\
&= \int_{\partial \Omega} g \overline{(\Pib\deltan)} \, \ds - \an(u_I,\deltan) - \im\k \int_{\partial \Omega} (\Pib u_I) \overline{(\Pib\deltan)} \, \ds \\
&= \int_{\partial \Omega} (\Pib g) \overline{\deltan} \, \ds - \sum_{K \in \taun} \anE(u_I,\deltan) - \im\k \int_{\partial \Omega} (\Pib u_I) \overline{(\Pib\deltan)} \, \ds \\
&= \int_{\partial \Omega} (\Pib g-g) \overline{\deltan} \, \ds + \int_{\partial \Omega} g \overline{\deltan} \, \ds - \sum_{K \in \taun} \anE(u_I-\uPW,\deltan) - \sum_{K \in \taun} a^K(\uPW,\deltan) \\
&\quad \quad- \im\k \int_{\partial \Omega} (\Pib u_I) \overline{\deltan} \, \ds,\\
\end{split}
\]
where~$\uPW \in \PWtaun$; whence, by applying the identity~\eqref{exact discrete method}, we get
\[
\begin{split}
\bn(\deltan,\deltan) 	
&= \int_{\partial \Omega} (\Pib g-g) \overline{\deltan} \, \ds + \sum_{K \in \taun}  a^K(u,\deltan) - \Nn(u,\deltan) + \im\k \int_{\partial \Omega} u \overline{\deltan} \, \ds \\ 
&\quad \quad - \sum_{K \in \taun} \anE(u_I-\uPW,\deltan) - \sum_{K \in \taun} a^K(\uPW,\deltan) - \im\k \int_{\partial \Omega} (\Pib u_I) \overline{\deltan} \, \ds \\
& = \sum_{\E \in \taun} \aE(u - \uPW, \deltan) - \sum_{\E \in \taun} \anE(\uI-\uPW,\deltan) +  \im\k \int_{\partial \Omega} (u - \Pib \uI) \overline {\deltan} \, \ds \\
&\quad \quad + \int_{\partial \Omega} (\Pib \g - \g) \overline{\deltan} \, \ds -\Nn(u,\deltan) =: R_1 + R_2 + R_3 + R_4 + R_5.
\end{split}
\]
We note that
\begin{equation} \label{Reb bound five terms}
I = \Re\left[ \bn(\deltan, \deltan) \right] \le \vert \bn(\deltan,\deltan) \vert \le \vert R_1 \vert + \vert R_2 \vert + \vert R_3 \vert + \vert R_4 \vert + \vert R_5 \vert,
\end{equation}
and we proceed  by bounding each of the five terms appearing on the right-hand side of~\eqref{Reb bound five terms}.
The term $R_1$ can be bounded by using the continuity of the local continuous sesquilinear forms:
\begin{equation} \label{bound R_1}
\vert R_1 \vert = \bigg| \sum_{\E \in \taun} \aE(u - \uPW, \deltan) \bigg| \le \Vert u - \uPW \Vert_{1,\k,\taun} \Vert \deltan \Vert_{1,\k,\taun}.
\end{equation}
For $R_2$, we make use of the local discrete continuity assumption~\eqref{theorem ass continuity}:
\begin{equation} \label{bound R_2}
\begin{split}
\vert R_2 \vert &\le \bigg| \sum_{\E \in \taun} \anE(\uI-\uPW,\deltan) \bigg| \le \gamma_h \Vert \uI - \uPW \Vert_{1,\k,\taun} \Vert \deltan \Vert_{1,\k, \taun}  \\
&\le \gamma_h \left\{ \Vert u - \uI \Vert_{1,\k,\taun} + \Vert u - \uPW \Vert_{1,\k,\taun} \right\} \Vert \deltan \Vert_{1,\k,\taun}.
\end{split}
\end{equation}
For the term $R_3$, we consider the following splitting:
\begin{equation} \label{splitting of R_3}
\begin{split}
R_3 &= \im\k \int_{\partial \Omega} (u - \Pib \uI) \overline{\deltan} \, \ds = \im\k \int_{\partial \Omega} (u - \Pib u) \overline{\deltan} \, \ds + \im\k \int_{\partial \Omega} \Pib (u - \uI) \overline{\deltan} \, \ds \\
&=: R_3^A+R_3^B.
\end{split}
\end{equation}
By using the definition and the properties of the $L^2$ projector $\Pib$ in~\eqref{projector boundary L2}, and by applying the $L^2(e)$, for all $e \in \Enb$, and the $\ell^2$ Cauchy-Schwarz inequalities, we derive
\begin{equation} \label{initial bound R_3}
\begin{split}
\vert R_3^A \vert &= \left \vert \im\k \int_{\partial \Omega} (u-\Pib u)\overline{\deltan} \, \ds \right \vert 
= \left \vert \im\k \sum_{\e \in \Enb} \int_{\e} (u-\Pie u)\overline{(\deltan-c)} \, \ds \right \vert \\
&\le \k \sum_{\e \in \Enb} \Vert u - \Pie u \Vert_{0,\e} \Vert \deltan - c \Vert_{0,\e}
\le \k \left( \sum_{\e \in \Enb} \Vert u - \Pie u \Vert^2_{0,\e} \right)^{\frac{1}{2}} \left(\sum_{\e \in \Enb} \Vert \deltan - c \Vert^2_{0,\e} \right)^{\frac{1}{2}},
\end{split}
\end{equation}
for any edgewise complex constant function $c$.

We bound the two terms on the right-hand side of~\eqref{initial bound R_3} as follows.
Given $e \in \Enb$ an arbitrary edge on $\partial \Omega$, using~\eqref{bound u-Piu} and the definition of the norm $\lVert \cdot \rVert_{1,k,\E}$, we have
\begin{equation*} 
\Vert u - \Pie u \Vert_{0,\e} \lesssim \hE^{\frac{1}{2}} \vert u - \wE \vert_{1,\E} \le \hE^{\frac{1}{2}} \Vert u - \wE \Vert_{1,\k,\E} \quad \forall \wE \in \PWE,
\end{equation*}
where $\E$ is the unique polygon in $\taun$ such that $\e \in \EE \cap \Enb$.

Concerning the second term on the right-hand side of~\eqref{initial bound R_3}, we make use of the trace inequality~\eqref{trace inequality} and of the Poincar\'e-Friedrichs inequality~\eqref{Poincare Friedrich full}, choosing $c= \frac{1}{\vert \E \vert} \int_\E \deltan \, \dx \in \C$, to obtain
\begin{equation} \label{Poincare trace trick}
\rVert \deltan - c \lVert_{0,\e}  \le \rVert \deltan - c \lVert_{0,\partial \E} 
\lesssim \hE^{-\frac{1}{2}} \lVert \deltan -c \rVert_{0,\E} +  \hE^{\frac{1}{2}} \vert \deltan \vert_{1,\E} 
\lesssim \hE^{\frac{1}{2}} \vert \deltan \vert_{1,\E} 
\le \hE^{\frac{1}{2}} \Vert \deltan \Vert_{1,\k,\E}.
\end{equation}
Thus,
\begin{equation} \label{estimate R_3^A}
\vert R_3^A \vert
\lesssim \k\h \lVert u - \uPW \rVert_{1,\k, \taun} \Vert \deltan \Vert_{1,k,\taun} \quad \forall \uPW \in \PWtaun.
\end{equation}
The term $R_3^B$ on the right-hand side of~\eqref{splitting of R_3} can be bounded in a similar fashion. More precisely, we first note that, due to the definitions of $\Pie$ in $\eqref{projector edge L2}$ and of $\uI$ in~\eqref{definition TVEM interpolant}, we have
\begin{equation} \label{madness}
\int_{\e} \Pie(u-\uI) \, \overline{c} \, \ds =\int_{\e} (u-\uI) \, \overline{c} \, \ds =0,
\end{equation}
for all complex constant functions $c$. Owing to~\eqref{madness}, it follows
\begin{equation} \label{R_3^B}
\begin{split}
\vert R_3^B\vert &= \k \left\vert \int_{\partial \Omega} \Pib(u-\uI) \overline{\deltan} \, \ds \right\vert 
= \k \left\vert \sum_{\e \in \Enb} \int_{\e} \Pie(u-\uI) \overline{(\deltan - c)} \, \ds \right\vert \\
&\le \k \sum_{\e \in \Enb} \Vert u - \uI \Vert_{0,\e} \Vert \deltan-c \Vert_{0,\e}
\le \k \left( \sum_{\e \in \Enb} \Vert u - \uI \Vert^2_{0,\e} \right)^{\frac{1}{2}} \left(\sum_{\e \in \Enb} \Vert \deltan - c \Vert^2_{0,\e} \right)^{\frac{1}{2}},
\end{split}
\end{equation} 
where the stability of the projector $\Pie$ in the $L^2(\e)$ norm was used in the first inequality.

Due to the trace inequality~\eqref{trace inequality}, the Poincar\'e-Friedrichs inequality for $u-\uI$ in~\eqref{Poincare inequality}, and the definition of the norm $\Vert \cdot \Vert_{1,\k,\E}$, we obtain
\[
\Vert u - \uI \Vert_{0,\e} \le \Vert u - \uI \Vert_{0,\partial \E} \lesssim \hE^{\frac{1}{2}} \vert u - \uI \vert_{1,\E} \le \hE^{\frac{1}{2}} \Vert u - \uI \Vert_{1,\k,\E}.
\]
The term $\Vert \deltan - c \Vert^2_{0,\e}$ in~\eqref{R_3^B} can be estimated, for all boundary edges $\e \in \Enb$, as in~\eqref{Poincare trace trick}, by fixing $c$ as the average of $\deltan$ over $\E$, with $\E$ being the unique polygon in $\taun$ such that $\e \in \EE \cap \Enb$.
After inserting these estimates into~\eqref{R_3^B}, one gets
\begin{equation} \label{estimate R_3^B}
\vert R_3^B \vert
\lesssim \k\h \lVert u - \uI \rVert_{1,\k, \taun} \Vert \deltan \Vert_{1,k,\taun}.
\end{equation}
Combining~\eqref{estimate R_3^A} and~\eqref{estimate R_3^B} leads to
\begin{equation} \label{bound R_3}
\vert R_3 \vert \lesssim \k\h \left( \Vert u - \uPW \Vert_{1,\k, \taun} +\Vert u - \uI \Vert_{1,\k,\taun}\right) \Vert \deltan \Vert_{1,k,\taun}.
\end{equation}
For the term $R_4$, by mimicking what was done in~\eqref{initial bound R_3} and~\eqref{Poincare trace trick}, i.e., making appear an edgewise constant on $\partial \Omega$ and using the Poincar\'e inequality, we get
\begin{equation} \label{bound R_4}
\vert R_4 \vert = \left\vert \int_{\partial \Omega} (\g- \Pib \g) \overline{\deltan} \, \ds \right\vert 
\lesssim \h^{\frac{1}{2}} \Vert \g- \Pib \g \Vert_{0,\partial \Omega} \Vert \deltan \Vert_{1,\k,\taun}.
\end{equation}
Finally, we have to study the nonconformity term $R_5$ on the right-hand side of~\eqref{Reb bound five terms}.

Using the definitions of the nonconforming space $\Vnp$ in~\eqref{global Trefftz space} and of the projector $\Pie$ in~\eqref{projector edge L2}, together with the Cauchy-Schwarz inequality, yield
\begin{equation} \label{initial bound R_5}
\begin{split}
\vert R_5 \vert &= \left\vert \Nn (u, \deltan) \right\vert 
= \left\vert \sum_{\e \in \mathcal E^I_n} \int_\e \nabla u \cdot \overline{\llbracket \deltan \rrbracket} \, \ds \right\vert
= \left\vert \sum_{\e \in \mathcal E_n^I} \int_\e (\nabla u - \Pie (\nabla u)) \cdot \normale \overline{(\deltan^+ - \deltan^- )} \, \ds \right\vert  \\
&\le \left\vert \sum_{\e \in \mathcal E_n^I} \int_\e (\nabla u - \Pie (\nabla u)) \cdot \normale \overline{(\deltan^+  -{c}^+)} \, \ds \right\vert
+\left\vert \sum_{\e \in \mathcal E_n^I} \int_\e (\nabla u - \Pie (\nabla u)) \cdot \normale \overline{(\deltan^-  - {c}^-)} \, \ds \right\vert \\
& \le  \left( \sum_{\e \in \EnI} \Vert \nabla u \cdot \normale - \Pie(\nabla u \cdot \normale) \Vert^2_{0,\e}   \right)^{\frac{1}{2}} \left( \sum_{\e \in \EnI} \Vert \delta_\h^+  - c^+  \Vert^2_{0,\e}   \right)^{\frac{1}{2}}\\
& \quad  + \left( \sum_{\e \in \EnI} \Vert \nabla u \cdot \normale - \Pie(\nabla u \cdot \normale) \Vert^2_{0,\e}   \right)^{\frac{1}{2}} \left( \sum_{\e \in \EnI} \Vert \delta_\h^-  - c^-  \Vert^2_{0,\e}   \right)^{\frac{1}{2}},\\
\end{split}
\end{equation}
for any edgewise complex constant functions ${c}^+$ and ${c}^- \in \C$.

After applying the Cauchy-Schwarz inequality to both terms on the right-hand side of~\eqref{initial bound R_5}, we bound the resulting terms as follows. We begin with the bounds on the terms involving $\nabla u$.
Denoting by $\E^+$ and $\E^-$ the two polygons in $\taun$ with $\e \in \mathcal{E}^{\E^+} \cap \mathcal{E}^{\E^-}$,
owing to the trace inequality~\eqref{trace inequality} and the inequality~\eqref{Poincare Friedrich full}
for any $\wEDiamond \in \PW (\E^\pm)$ and $(\mathbf{c}^\pm)_i=\frac{1}{\vert \E^\pm \vert} \int_{\E^\pm} \nabla(u-\wEDiamond) \, \dx$, $i=1,2$, it holds
\begin{equation*} 
\begin{split}
\sum_{\e \in \EnI}   \Vert \nabla u \cdot \normale - \Pie (\nabla u \cdot \normale) \Vert^2_{0,\e} 
&\le \sum_{\e \in \EnI} \Vert (\nabla u - \nabla \wEDiamond  - \textbf{c}^\pm)\cdot \normale \Vert^2_{0,\e}  \\
& \le \sum_{\E \in \taun} \Vert (\nabla u -\nabla \wEDiamond -\textbf{c}^\pm) \cdot \normale \Vert_{0,\partial \E}^2 \\
&\lesssim \sum_{\E \in \taun} \hE \vert u  - \w \vert^2_{2,\E}.
\end{split}
\end{equation*}
For the terms involving $\delta_\h$, we take ${c}^\pm=\frac{1}{\vert \E^\pm \vert} \int_{\E^\pm} \deltan^{\pm} \, \dx$, follow the computations in~\eqref{Poincare trace trick}, and obtain
\begin{equation*}
\Vert \deltan^{\pm}   - {c}^{\pm} \Vert^2_{0,\e}
\lesssim \h \Vert \deltan^{\pm} \Vert^2_{1,\k,\E^{\pm}},
\end{equation*}
Thus, a bound on the nonconformity term $R_5$ is given by
\begin{equation} \label{bound R_5}
\vert R_5 \vert = \left\vert \Nn (u, \deltan)  \right\vert \lesssim \h \vert u - \uPW \vert_{2,\taun} \Vert \deltan \Vert_{1,\k,\taun} \quad \forall \uPW \in \PWtaun.
\end{equation}
Collecting~\eqref{bound R_1}, \eqref{bound R_2}, \eqref{bound R_3}, \eqref{bound R_4}, and~\eqref{bound R_5} \eqref{Reb bound five terms}, we get
\begin{equation} \label{bound I}
\begin{split}
I = \Re \left[ \bn(\deltan,\deltan) \right]
& \lesssim \{ ( \k\h+\gamma_h+1) \Vert u - \uPW \Vert_{1,\k,\taun} + (\k\h+\gamma_h) \Vert u - \uI \Vert_{1,\k,\taun} \\
& \quad \quad + \h^{\frac{1}{2}} \Vert \g- \Pie \g \Vert_{0,\partial \Omega} + \h \vert u - \uPW \vert_{2, \taun} \} \Vert \deltan \Vert_{1,\k,\taun}.
\end{split}
\end{equation}
\textit{Step~3: Estimate of the term~$II$ in~\eqref{applying Garding}:}
By using simple algebra and the definitions of $\deltan$ and the norm $\lVert \cdot \rVert_{1,\k,\taun}$, we obtain
\begin{equation} \label{bound II}
\begin{split}
II & = 2\k^2 \Vert \deltan \Vert^2_{0,\Omega}  
= 2\k^2 \Vert \un - \uI \Vert_{0,\Omega} \Vert \deltan \Vert_{0,\Omega}
\le  2\k^2 \left\{ \Vert u - \un \Vert_{0,\Omega} + \Vert u - \uI \Vert_{0,\Omega}  \right\} \Vert \deltan \Vert_{0,\Omega} \\
& \le 2 \left\{ \k \Vert u - \un \Vert_{0,\Omega} + \Vert u - \uI \Vert_{1,\k,\taun} \right\} \Vert \deltan \Vert_{1,\k,\taun}.
\end{split}
\end{equation}
We plug~\eqref{bound I} and~\eqref{bound II} into~\eqref{applying Garding} and divide by $\Vert \deltan \Vert_{1,\k,\taun}$, deducing
\begin{equation} \label{abstract analysis before duality}
\begin{split}
\alpha_h \Vert \un - \uI \Vert_{1,\k,\taun} 	
& \lesssim (\k\h+\gamma_h+1) \Vert u - \uPW \Vert_{1,\k,\taun} + (\k\h+\gamma_h + 1) \Vert u - \uI \Vert_{1,\k,\taun} \\
& \quad + h^{\frac{1}{2}} \Vert \g- \Pie\g \Vert_{0,\partial \Omega} + \h \vert u - \uPW \vert_{2, \taun} + \k \Vert u - \un\Vert_{0,\Omega}.
\end{split}
\end{equation}
\textit{Step~4: Estimate of~$\lVert u-\un \rVert_{0,\Omega}$:}
We consider the auxiliary dual problem: find $\psi$ such that
\begin{equation} \label{dual problem}
\left\{
\begin{alignedat}{2}
-\Delta \psi - \k^2\psi &= u - \un  &&\quad \text{in } \Omega\\
\nabla \psi \cdot \n_\Omega - \im\k \psi&=0  &&\quad \text{on } \partial \Omega.
\end{alignedat}
\right.
\end{equation}
The convexity of $\Omega$ and \cite[Proposition 8.1.4]{melenk_phdthesis} imply that the solution $\psi$ to the weak formulation of~\eqref{dual problem} belongs to $H^2(\Omega)$ and that the stability bounds 
\begin{equation} \label{Melenk a priori}
\Vert \psi \Vert _{1,\k,\taun} \le \dOmega \Vert u - \un \Vert_{0,\Omega},\quad \quad \vert \psi \vert_{2,\Omega} \lesssim (1+\dOmega\k) \Vert u - \un \Vert_{0,\Omega}
\end{equation}
are valid, with $\dOmega$ being a positive universal constant depending only on  $\Omega$.

In addition, for all $\E \in \taun$, there exists $\psiE \in \PWE$ such that, see \cite[Propositions 3.12 and 3.13]{GHP_PWDGFEM_hversion},
\begin{equation} \label{best approx psi-psiPW}
\lVert \psi-\psiE \rVert_{0,\E} \lesssim \hE^2 (\vert \psi \vert_{2,\E} + \k^2 \Vert \psi \Vert_{0,\E} ), \quad \lvert \psi-\psiE \rvert_{1,\E} \lesssim \hE (\k \hE+1) (\vert \psi \vert_{2,\E} + \k^2 \Vert \psi \Vert_{0,\E} ),
\end{equation}
where the hidden constants depend only on the shape of the element $\E$ and on $\p$.

Hence, combining~\eqref{best approx psi-psiPW} with~\eqref{Melenk a priori}, there exists~$\psiPW \in \PWtaun$ such that
\begin{equation} \label{pw best approximation}
\Vert \psi - \psiPW \Vert_{1,\k,\taun} \lesssim \h(1+\h\k)(1+\dOmega\k) \Vert u - \un \Vert_{0,\Omega} =: \varsigma(\k,\h) \Vert u - \un \Vert_{0,\Omega},
\end{equation}
where the hidden constant is independent of $\h$, $\k$, and $\psi$.

Besides, thanks to Theorem \ref{theorem best approximation result}, together with~\eqref{Melenk a priori} and~\eqref{best approx psi-psiPW},
defining the ``interpolant'' $\psiI$ of~$\psi$ as in~\eqref{definition TVEM interpolant},
\begin{equation} \label{Trefftz best approximation}
\begin{split}
\Vert \psi 	- \psiI \Vert_{1,\k,\taun} \le \cBA(\k\h) \lVert \psi - \psiPW \rVert_{1,\k,\taun} 
\lesssim \vartheta(\k,\h) \Vert u - \un \Vert_{0,\Omega},
\end{split}
\end{equation}
where $\cBA(\k\hE)$ and $\vartheta(\k,\h)$ are defined in~\eqref{def cBA} and~\eqref{definition beta}, respectively.

From the definition of the dual problem~\eqref{dual problem}, integrating by parts and using the definition of $\Nn(\cdot,\cdot)$ in~\eqref{definition Nn}, we get
\begin{equation} \label{bound L2 4 terms}
\begin{split}
&\Vert u - \un \Vert^2_{0,\Omega}	 = \sum_{\E \in \taun} \int_\E (-\Delta \psi - \k^2\psi) \overline{(u - \un)} \, \dx  \\
							& = \sum_{\E \in \taun} \left(   \int_\E \left( \nabla \psi \cdot \overline{\nabla(u - \un)} - \k^2 \psi \overline{(u-\un)}  \right) \, \dx - \int_{\partial \E} (\nabla \psi \cdot \nE) \overline{(u - \un)} \, \ds \right) \\
							& = \sum_{\E \in \taun} \aE(\psi,u-\un) - \Nn(\psi, u - \un) - \int_{\partial \Omega} (\nabla \psi \cdot \n_\Omega) \overline{(u-\un)}\, \ds  \\
							& = \sum_{\E \in \taun} \aE(\psi-\psiI,u-\un) + \sum_{\E \in \taun} \aE(\psiI,u-\un) - \Nn(\psi, u - \un) \\
							&\quad - \im\k \int_{\partial \Omega} \psi \overline{(u-\un)} \, \ds  =: S_1 + S_2 + S_3 + S_4.
\end{split}
\end{equation}
Hence, we need to bound the four terms on the right-hand side of~\eqref{bound L2 4 terms}.
We begin with $S_1$. Using the continuity of the continuous local sesquilinear forms, together with~\eqref{Trefftz best approximation}, we have
\begin{equation} \label{bound (1)}
\begin{split}
\vert S_1 \vert = \bigg| \sum_{\E \in \taun} \aE(\psi-\psiI,u-\un) \bigg| 	& \le \Vert \psi - \psiI \Vert_{1,\k,\taun} \Vert u - \un \Vert_{1,\k,\taun} \\
														& \lesssim \vartheta(\k,\h) \Vert u - \un \Vert_{0,\Omega} \Vert u - \un \Vert_{1,\k,\taun}.\\
\end{split}
\end{equation}
The nonconformity term $S_3$ can be bounded analogously as the term $R_5$ in~\eqref{bound R_5}. By taking the special choice~$w^{\taun}={0}$ and using~\eqref{Melenk a priori}, we arrive at
\begin{equation} \label{bound (3)}
|S_3| = |\Nn(\psi, u - \un)| \lesssim \h(1+\dOmega \k) \Vert u - \un \Vert_{0,\Omega} \Vert u - \un \Vert_{1,\k,\taun}.
\end{equation}
It remains to control the terms $S_2$ and $S_4$. For $S_2$, we observe that using the identity~\eqref{exact discrete method} and taking the complex conjugated of~\eqref{HH problem weak discrete formulation}, with the definitions~\eqref{definition bn} and~\eqref{discrete rhs}, give
\[
\begin{split}
S_2 & = \sum_{\E \in \taun} \aE(\psiI, u - \un) 
= \sum_{\E \in \taun} \overline{\aE(u,\psiI)} -\sum_{\E \in \taun} \aE(\psiI,\un)\\
& = \overline{\Nn(u,\psiI)} + \int_{\partial \Omega} \overline{\g} \psiI \, \ds + \im\k \int_{\partial \Omega} \overline{u} \psiI \, \ds + \sum_{\E \in \taun} \left\{ -\anE(\psiI, \un) + \anE(\psiI,\un) - \aE(\psiI, \un) \right\}  \\
& = \overline{\Nn(u,\psiI)} + \int_{\partial \Omega} \overline{\g} (\psiI - \Pib \psiI) \, \ds + \im\k \int_{\partial \Omega} \overline{(u - \Pib \un)} \psiI \, \ds \\
&\quad \quad + \sum_{\E\in \taun} \left\{ \anE(\psiI,\un) - \aE(\psiI, \un)  \right\}.
\end{split}
\]
We deduce
\begin{equation} \label{bound (2)+(4)}
\begin{split}
S_2+S_4 & = \overline {\Nn (u,\psiI)} + \int_{\partial \Omega}  \overline{\g} (\psiI - \Pib \psiI) \, \ds + \im\k \left( \int_{\partial \Omega} \psiI \overline{(u-\Pib \un)} \, \ds - \int_{\partial \Omega} \psi \overline{(u-\un)} \, \ds \right)  \\
& \quad + \sum_{\E \in \taun} \left\{ \anE(\psiI, \un) - \aE(\psiI, \un)  \right\} =: T_1 + T_2 + T_3 + T_4.
\end{split}
\end{equation}
The term $T_1$ can be bounded by using~\eqref{bound R_5}, \eqref{Melenk a priori}, and~\eqref{Trefftz best approximation}:
\begin{equation} \label{bound T_1}
\begin{split}
\vert T_1 \vert &= |\overline{\Nn (u,\psiI)}| 
\le \h \vert u - \uPW \vert_{2, \taun} \Vert \psiI \Vert_{1,\k,\taun} \\
&\le \h \vert u - \uPW \vert_{2, \taun} \left( \Vert \psi \Vert_{1,\k,\taun} + \Vert \psi-\psiI \Vert_{1,\k,\taun} \right) \\
&\lesssim \h (1 + \vartheta(\k,\h)) \vert u - \uPW \vert_{2,\taun} \Vert u - \un \Vert_{0,\Omega}
\end{split}
\end{equation}
for any $\uPW \in \PWtaun$.

For $T_2$, we observe that, with the definition of the projector $\Pie$ given in~\eqref{projector edge L2}, it follows
\[
\vert T_2 \vert =\left\vert\int_{\partial \Omega} \g \overline{(\psiI - \Pib \psiI)} \, \ds \right\vert = \left\vert\int_{\partial \Omega} (\g- \Pib \g) \overline{(\psiI - c)} \, \ds \right\vert,
\]
where $c$ is any edgewise complex constant.

By doing similar computations as in~\eqref{initial bound R_3} and~\eqref{Poincare trace trick}, and employing also~\eqref{Melenk a priori} and~\eqref{Trefftz best approximation},  we get
\begin{equation} \label{bound T_2}
\begin{split}
\vert T_2 \vert &\lesssim \h^{\frac{1}{2}} \Vert \g - \Pib \g \Vert_{0,\partial \Omega} \Vert \psiI \Vert_{1,\k,\taun} 
\le \h^{\frac{1}{2}} \Vert \g - \Pib \g \Vert_{0,\partial \Omega} \left( \Vert \psi \Vert_{1,\k,\taun} + \Vert \psi-\psiI \Vert_{1,\k,\taun} \right) \\
&\le \h^{\frac{1}{2}} (1 + \vartheta(\k,\h)) \Vert \g - \Pib \g \Vert_{0,\partial \Omega} \Vert u - \un \Vert_{0,\Omega}.
\end{split}
\end{equation}
The term $T_4$ can be bounded using the plane wave consistency property~\eqref{pw consistency}, the continuity of the sesquilinear forms $\an(\cdot,\cdot)$ and $\aE(\cdot,\cdot)$, and the approximation estimates~\eqref{pw best approximation} and~\eqref{Trefftz best approximation}:
\begin{equation} \label{bound T_4}
\begin{split}
\vert T_4 \vert
&= \left| \sum_{\E \in \taun} \left\{ \anE(\un,\psiI) - \aE(\un,\psiI)  \right\} \right| \\
&\le \sum_{\E \in \taun} \left| \anE(\un-\uPW , \psiI - \psiPW) - \aE(\un - \uPW, \psiI - \psiPW) \right|\\
& \le (\gamma_h+1) \Vert \un - \uPW \Vert_{1,\k,\taun} \Vert \psiI - \psiPW \Vert_{1,\k,\taun}  \\
& \le (\gamma_h+1) \left\{ \Vert u - \uPW \Vert_{1,\k,\taun} + \Vert u- \un \Vert_{1,\k,\taun} \right\}  \left\{ \Vert \psi - \psiI \Vert_{1,\k,\taun} + \Vert \psi- \psiPW \Vert_{1,\k,\taun} \right\} \\
& \lesssim (\gamma_h+1) \left\{ \Vert u - \uPW \Vert_{1,\k,\taun} + \Vert u- \un \Vert_{1,\k,\taun} \right\} \{ \vartheta(\k,\h)+\varsigma(\k,\h) \} \Vert u - \un \Vert_{0,\Omega}
\end{split}
\end{equation}
for all $\uPW, \psiPW \in \PWtaun$.

Finally, we bound $T_3$. We compute
\begin{equation*}
\begin{split}
\vert T_3 \vert &= \k \left\vert \int_{\partial \Omega} \psi \overline{(u-\un)} \, \ds - \int_{\partial \Omega} \psiI \overline{(u-\Pib \un)} \, \ds \right\vert \\
&= \k \left\vert \int_{\partial \Omega} (\psi - \psiI) \overline{(u-\un)} \, \ds - \int_{\partial \Omega} \psiI \overline{(\un - \Pib \un)} \, \ds \right\vert.
\end{split}
\end{equation*}
Using the definitions of $\psiI$ as in~\eqref{definition TVEM interpolant} and of $\Pib$ in~\eqref{projector boundary L2}, we obtain
\begin{equation} \label{first bound T_3}
\begin{split}
\vert T_3 \vert & = \k \left\vert \int_{\partial \Omega} (\psi-\psiI) \overline{(u-\un - \Pib(u-\un))} \, \ds - \int_{\partial \Omega} (\psiI - \Pib \psiI) \overline{(\un - \Pib \un)} \, \ds \right\vert \\
& = \k \bigg| \int_{\partial \Omega} (\psi-\psiI) \overline{(u-\un - \Pib(u-\un))} \, \ds 
- \int_{\partial \Omega} (\psiI - \Pib \psiI) \overline{(\un-u)} \, \ds  \\
& \quad \quad - \int_{\partial \Omega} (\psiI - \Pib \psiI) \overline{(u - \Pib u)}  \, \ds \bigg| \\
& =: \k \vert T_3^A - T_3^B - T_3^C \vert \le \k \left( \vert T_3^A\vert + \vert T_3^B\vert +\vert T_3^C\vert  \right).
\end{split}
\end{equation}
We bound the three terms on the right-hand side of~\eqref{first bound T_3} with tools analogous to those employed so far.
The term $T_3^A$ can be estimated using the Cauchy-Schwarz inequality, the trace inequality~\eqref{trace inequality},
the definition of $\psiI$, the Poincar\'{e}-Friedrichs inequality~\eqref{Poincare inequality}, and the identity~\eqref{bound u-Piu} with $\wE=0$:
\begin{equation} \label{bound Ttilde1}
\begin{split}
\vert T_3^A \vert 	& = \left\vert  \int_{\partial \Omega} (\psi - \psiI) \overline{(u-\un - \Pib(u-\un))} \, \ds \right\vert \le \Vert \psi - \psiI \Vert_{0,\partial \Omega} \Vert u-\un - \Pib(u-\un) \Vert_{0,\partial \Omega} \\
& \lesssim \h \lVert \psi - \psiI \rVert_{1,\k,\taun} \lVert u -\un \rVert_{1,\k,\taun}.
\end{split}
\end{equation}
For $T_3^B$, we can do analogous computations as in~\eqref{initial bound R_3} and~\eqref{Poincare trace trick}, getting
\begin{equation} \label{bound Ttilde2}
\begin{split}
\vert T_3^B \vert 	& = \left\vert  \int_{\partial \Omega} (\psiI - \Pib \psiI) \overline{(u-\un)} \, \ds \right\vert 
\lesssim \h \lVert \psiI - \psiPW \rVert_{1,\k,\taun} \lVert u - \un \rVert_{1,\k,\taun} \\
& \le \h (\lVert \psi - \psiI \rVert_{1,\k,\taun} + \lVert \psi - \psiPW \rVert_{1,\k,\taun}) \lVert u - \un \rVert_{1,\k,\taun} \quad \forall  \psiPW \in \PWtaun.
\end{split}
\end{equation}
The term $T_3^C$ is bounded by using~\eqref{bound u-Piu}:
\begin{equation} \label{bound Ttilde3}
\begin{split}
\vert T_3^C \vert 	& = \left\vert  \int_{\partial \Omega} (\psiI - \Pib \psiI) \overline{(u - \Pib u)} \, \ds  \right\vert 
\lesssim \h \lVert \psiI - \psiPW \rVert_{1,\k,\taun} \lVert u - \uPW \rVert_{1,\k,\taun} \\
& \le \h (\lVert \psi - \psiI \rVert_{1,\k,\taun} + \lVert \psi - \psiPW \rVert_{1,\k,\taun}) \lVert u - \uPW \rVert_{1,\k,\taun} \quad \forall \uPW,\, \psiPW \in \PWtaun.
\end{split}
\end{equation}
Plugging~\eqref{bound Ttilde1}, \eqref{bound Ttilde2}, and~\eqref{bound Ttilde3} in~\eqref{first bound T_3}, and using the approximation properties~\eqref{pw best approximation} and~\eqref{Trefftz best approximation}, yield
\begin{equation} \label{bound T_3}
\begin{split}
\vert T_3 \vert &\lesssim  \k\h (\lVert \psi - \psiPW \rVert_{1,\k,\taun} + \lVert \psi - \psiI \rVert_{1,\k,\taun}) \left( \lVert u - \uPW \rVert_{1,\k,\taun} + \lVert u - \un \rVert_{1,\k,\taun}\right)  \\
&\lesssim \k\h \left(  \lVert u - \uPW \rVert_{1,\k,\taun} + \lVert u - \un \rVert_{1,\k,\taun} \right) (\varsigma(\k,\h)+\vartheta(\k,\h)) \Vert u - \un\Vert_{0,\Omega}.
\end{split}
\end{equation}
Collecting and inserting~\eqref{bound T_1}, \eqref{bound T_2}, \eqref{bound T_4}, and~\eqref{bound T_3} into~\eqref{bound (2)+(4)}, we obtain the following bound:
\begin{equation} \label{final bound (2)+(4)}
\begin{split}
|S_2+S_4| 
&\lesssim \big\{ (1 + \vartheta(\k,\h)) (\h \vert u - \uPW \vert_{2,\taun} + \h^{\frac{1}{2}} \Vert \g - \Pie \g \Vert_{0,\partial \Omega})  \\
&\quad +(\gamma_h+1+\k\h) (\varsigma(\k,\h)+\vartheta(\k,\h)) \left[ \Vert u - \uPW \Vert_{1,\k,\taun} + \Vert u- \un \Vert_{1,\k,\taun} \right] \big\} \Vert u - \un \Vert_{0,\Omega}
\end{split}
\end{equation}
for all $\uPW \in \PWtaun$.

Plugging next~\eqref{bound (1)}, \eqref{bound (3)}, and~\eqref{final bound (2)+(4)} into~\eqref{bound L2 4 terms}, and dividing by $\Vert u - \un \Vert_{0,\Omega}$, we arrive at
\begin{equation} \label{final bound L2 4 terms}
\begin{split}
\Vert u - \un \Vert_{0,\Omega}
&\lesssim (1 + \vartheta(\k,\h)) (\h \vert u - \uPW \vert_{2,\taun} + \h^{\frac{1}{2}} \Vert \g - \Pie \g \Vert_{0,\partial \Omega})  \\
& \quad +(\gamma_h+1+\k\h) (\varsigma(\k,\h)+\vartheta(\k,\h)) \Vert u - \uPW \Vert_{1,\k,\taun} \\
& \quad +\left\{(\gamma_h+1+\k\h) (\varsigma(\k,\h)+ \vartheta(\k,\h)) + \h(1+\dOmega\k) \right\} \Vert u- \un \Vert_{1,\k,\taun}
\end{split}
\end{equation}
for all $\uPW \in \PWtaun$.\\
\textit{Step~5: Conclusion}:
We plug~\eqref{final bound L2 4 terms} in \eqref{abstract analysis before duality} and~\eqref{abstract analysis before duality} in~\eqref{triangular inequality in abstract analysis}, obtaining
\begin{equation} \label{final bound u-un}
\begin{split}
\Vert u - \un \Vert_{1,\k,\taun} 	
& \lesssim \frac{(\k\h+\gamma_h+1)(1 +  \k  (\varsigma(\k,\h)+\vartheta(\k,\h)))}{\alpha_h} \Vert u - \uPW \Vert_{1,\k,\taun} \\
&+ \left(\frac{(\k\h+\gamma_h+1)}{\alpha_h}+1 \right) \Vert u - \uI \Vert_{1,\k,\taun}  
+ \frac{(\k (1 + \vartheta(\k,\h)) + 1) \h}{\alpha_h} \vert u - \uPW \vert_{2,\taun} \\
& + \frac{(\k (1 + \vartheta(\k,\h))+1) \h^{\frac{1}{2}}}{\alpha_h} \Vert \g - \Pie \g \Vert_{0,\partial \Omega}\\
& + \frac{\k\left\{(\gamma_h+1+\k\h) (\varsigma(\k,\h) + \vartheta(\k,\h)) + \h(1+\dOmega \k) \right\}}{\alpha_h} \Vert u- \un \Vert_{1,\k,\taun}
\end{split}
\end{equation}
for all plane waves $\uPW \in \PWtaun$, where $\varsigma(\k,\h)$ and $\vartheta(\k,\h)$ are given in~\eqref{definition beta}.

Assuming that $\k^2 \h$ is sufficiently small, for instance, having set $\ctilde$ the hidden constant in~\eqref{final bound u-un},
\begin{equation} \label{h small condition}
\ctilde \, \frac{\k\left\{(\gamma_\h+1+\k\h) (\varsigma(\k,\h) + \vartheta(\k,\h)) + \h(1+\dOmega\k) \right\}}{\alpha_h} \le \frac{1}{\nu},
\end{equation}
for some $\nu>1$, we can bring the last term on the right-hand side of~\eqref{final bound u-un} to the left-hand side and obtain, further using~\eqref{best approximation u-uI}, the desired bound~\eqref{abstract error estimate}. 
\end{proof}

\subsection{\textit{A priori} error bounds} \label{subsection a priori error bound}
From Theorem \ref{theorem abstract theory}, we deduce \textit{a priori} error bounds in terms of $\h$. 
The best approximation terms with respect to plane waves on the right-hand side of~\eqref{abstract error estimate}, namely $\Vert u - \uPW \Vert_{1,\k,\taun}$ and $\vert u - \uPW \vert_{2,\taun}$,
can be bounded using Theorem \ref{theorem local h-best approximation plane waves}.
A bound for the third term, namely~$\Vert \g - \Pib \g \Vert_{0,\partial \Omega}$, is given in the following proposition.

\begin{prop} \label{prop proj error}
Let $\taun$ satisfy the assumptions (\textbf{G1})-(\textbf{G3}) and let $\{ \dl \}_{\ell=1,\dots,p}$, $\p = 2q+1$, $q \in \N_{\ge 2}$, be a given set of plane wave directions fulfilling the assumption (\textbf{D1}). 
Assuming that $\h$ is sufficiently small, see~\eqref{h small condition again} below, and given $\g$ defined on $\partial \Omega$ with $\g_\e:=\g_{|_\e} \in H^{s-\frac{1}{2}}(e)$ for all $\e \in \Enb$ and for some $s \in \R_{\ge 1}$, we have
\begin{equation*}
\Vert \g - \Pib \g \Vert_{0,\partial \Omega} 
\lesssim e^{\left( \frac{7}{4}-\frac{3}{4}\rho_{\max} \right)\sigma(\k\h)} \left( 1+[\sigma (k\h)]^{q+9} \right) \h^{\zeta+\frac{1}{2}} \sum_{e \in \Enb} \lVert G \rVert_{\zeta+1,\k,\De},
\end{equation*}
where $\zeta:=\min(q,s)$, $\Pib$ is defined in~\eqref{projector boundary L2}, the constant $\sigma>1$ with $\sigma \approx 1$, and $G$ and $\rho_{\max}$ are set in~\eqref{auxiliary problem} and~\eqref{rhomax} below, respectively. 
\end{prop}

\begin{proof}
Associated with every boundary edge $e \in \Enb$, we consider a domain $D_e$ with $C^\infty$-boundary and diameter $\h_{\De}=\sigma\h$, where~$\sigma>1$ is such that~$\hDe \approx \h$, and~$\De$ satisfies
\begin{itemize}
\item $e \in \partial \De$;
\item there exist $\rho_{\De} \in (0,\frac{1}{2}]$ and $0<\rho_{0,\De} \le \rho_{\De}$, such that the ball $B_{\rho_{\De} \hDe}$ is contained in $\De$, and $\De$ is star-shaped with respect to $B_{\rho_{0,\De} \hDe}$;
\item it holds that
\begin{equation} \label{h small condition again}
\text{$\k^2$ is not a Dirichlet-Laplace eigenvalue in $\De$},
\end{equation}
which means that $\k^2$ fulfils the counterpart of the condition~\eqref{condition for Dirichlet-Laplace} on $\De$.
\end{itemize}
A graphical example of $\De$ with smooth boundary is provided in Figure \ref{fig:domainde}.
The construction of such domains is based on convolution techniques, as done in~\cite{smoothing_polygons}.
\begin{figure}[h]
\begin{center}
\begin{tikzpicture}
\draw[-,orange,line width=2pt] ({3*cos(-36)},{3*sin(-36)}) -- ({3*cos(36)},{3*sin(36)});
\draw[-,line width=2pt] ({3*cos(36)},{3*sin(36)}) -- ({3*cos(108)},{3*sin(108)});
\draw[-,line width=2pt] ({3*cos(108)},{3*sin(108)}) -- ({3*cos(180)},{3*sin(180)});
\draw[-,line width=2pt] ({3*cos(180)},{3*sin(180)}) -- ({3*cos(-108)},{3*sin(-108)});
\draw[-,line width=2pt] ({3*cos(-108)},{3*sin(-108)}) -- ({3*cos(-36)},{3*sin(-36)});
\draw[dashed,blue,line width=2pt] ({3*cos(-36)-1},{3*sin(-36)}) -- ({3*cos(36)-1},{3*sin(36)});
\draw [dashed,blue,line width=2pt,domain=0:180] plot ({3*cos(-36)-0.5+0.5*cos(\x)}, {3*sin(36)+0.5*sin(\x)}); 
\draw [dashed,blue,line width=2pt,domain=180:360] plot ({3*cos(-36)-0.5+0.5*cos(\x)}, {3*sin(-36)+0.5*sin(\x)}); 
\coordinate[label=left:\Large ${K_e}$] (Ke) at (0,0);
\coordinate[label=left:\Large $\textcolor{blue}{\De}$] (De) at (1.2,1);
\coordinate[label=right:\Large $\textcolor{orange}{e}$] (e) at ({3*cos(36)+0.15},0);
\end{tikzpicture}
\end{center}
\caption{A possible construction for the domain $\De$ with smooth boundary, given a boundary edge $e \in \Enb \cap K_e$, for some polygon $K_e$ belonging to a mesh $\taun$.}
\label{fig:domainde}
\end{figure}
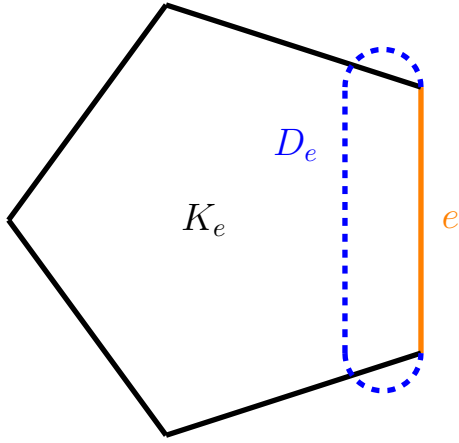

Note that the requirement on $\sigma$ guarantees a uniformly bounded overlapping of the collection of extended domains $\De$ associated with all the boundary edges $\e\in \Enb$.
More precisely,  there exists $N \in \N$ such that, for all $\x \in \R^2$, $\x$ belongs to the intersection of at most $N$ domains $\De$, $e \in \Enb$.
Owing to the smoothness of $\partial \De$, $e \in \Enb$, it is possible to extend $g_e$ to an $H^{s-\frac{1}{2}}(\partial \De)$ function, following e.g. \cite[Sect. 5.4]{evansPDE}, which we denote by~$\widetilde{g_e}$.
Note that~$\widetilde{\g_\e}_{|_\e}=\g_\e$.

Next, we consider the Helmholtz problem
\begin{equation} \label{auxiliary problem}
\left\{
\begin{alignedat}{2}
-\Delta G - \k^2 G &= 0  &&\quad \text{in } \De\\
G&=\widetilde{g_e}  &&\quad \text{on } \partial \De.
\end{alignedat}
\right.
\end{equation}
Well-posedness follows from the fact that $\k^2$ is not a Dirichlet-Laplace eigenvalue in $\De$, see~\eqref{h small condition again}.
Denoting by $\gamma^{-1}$ the continuous right-inverse trace operator~\cite[Theorem 3.37]{mclean2000strongly} and introducing $G_0:=G-\gamma^{-1} \widetilde{g_e}$, we can rewrite~\eqref{auxiliary problem} as a Helmholtz problem with zero Dirichlet boundary conditions:
\begin{equation*}
\left\{
\begin{alignedat}{2}
-\Delta G_0 - \k^2 G_0 &= f_0  &&\quad \text{in } \De\\
G_0&=0  &&\quad \text{on } \partial \De,
\end{alignedat}
\right.
\end{equation*}
with right-hand side $f_0:=(\Delta+\k^2) (\gamma^{-1} \widetilde{g_e}) \in H^{s-2}(\De)$.

Standard regularity theory \cite[Sect. 6.3]{evansPDE} implies $G_0 \in H^{s}(\De)$ and therefore $G \in H^{s}(\De)$. Then, by using the definition of the projector $\Pib$ in~\eqref{projector boundary L2} on every edge $e \in \Enb$, we obtain
\begin{equation*}
\Vert \g - \Pib \g \Vert_{0,\e} \le \Vert g_e - \wDe - c_e \Vert_{0,\e} \quad \forall \wDe \in \PWDe,\, \forall c_e \in \mathbb C.
\end{equation*}
By applying the trace inequality~\eqref{trace inequality}, selecting $c_e=\frac{1}{\vert \De \vert} \int_{\De} (G-w^{\De}) \, \dx$, and using the Poincar\'e inequality~\eqref{Poincare Friedrich full}, we get
\[
\Vert \g - \Pib \g \Vert_{0,\e}
\le \CT^{\frac{1}{2}} (\CP^2+1)^{\frac{1}{2}} \h^{\frac{1}{2}} \vert G- \wDe \vert_{1,\De} \quad \forall \wDe \in \PWDe.
\]
For $s=1$, this can be bounded by simply taking $\wDe=0$. Provided that $s \in \mathbb R_{>1}$, we can use Theorem \ref{theorem local h-best approximation plane waves} to get
\begin{equation*} 
\Vert \g - \Pib \g \Vert_{0,\e} 
\lesssim e^{\left( \frac{7}{4}-\frac{3}{4}\rho_{\De} \right)\k\hDe} \left( 1+(k\hDe)^{q+9} \right) \hDe^{\zeta+\frac{1}{2}} \lVert G \rVert_{\zeta+1,\k,\De},
\end{equation*}
where $\zeta:=\min(q,s-1)$, and the hidden constant is independent of $\k$, $\hDe$, and $G$. Defining
\begin{equation} \label{rhomax}
\rho_{\max}:=\max_{\De} \rho_{\De},
\end{equation} 
and summing over all edges $e \in \Enb$ give the desired result.
\end{proof}

The following theorem states the \textit{a priori} error estimate associated with the method~\eqref{HH problem weak discrete formulation}.

\begin{thm}
Let $u \in H^{s+1}(\Omega)$, $s \in \R_{\ge 1}$, be the exact solution to~\eqref{HH problem weak formulation}. Under the same assumptions as in Theorem \ref{theorem abstract theory} and Proposition \ref{prop proj error}, the following \textit{a priori} error bound is valid: 
\begin{equation*} 
\begin{split}
\Vert u - \un \Vert_{1,\k,\taun} 
&\lesssim \cPW (\k\h) \h^{\zeta_{1,2}} \left(\aleph_1(\k,\h) + \aleph_2(\k,\h)\right) \lVert u \rVert_{\zeta_{1,2}+1,\k,\taun} \\
& \quad + \h^{\zeta_3+1} \, \aleph_2(\k,\h) e^{\left( \frac{7}{4}-\frac{3}{4}\rho_{\max} \right)\sigma(\k\h)} \left( 1+[\sigma (k\h)]^{q+9} \right) \sum_{e \in \Enb} \lVert G \rVert_{\zeta_3+1,\k,\De},
\end{split}
\end{equation*}
where $\zeta_{1,2}:=\min(q,s)$, $\zeta_3:=\min(q,s-1)$, the constants $\cPW(\k\h)$, $\aleph_1(\k,\h)$, and $\aleph_2(\k,\h)$ are defined in~\eqref{cPW definition} and~\eqref{alephs}, respectively, and where $\rho_{\max}$, $\sigma$, $G$, and $\De$ are constructed in Proposition \ref{prop proj error}.
\end{thm}

\begin{proof}
The assertion follows directly by combining the abstract error estimate~\eqref{abstract error estimate} in Theorem~\ref{theorem abstract theory} with best approximation estimates.
More precisely, the first and second terms on the right-hand side of~\eqref{abstract error estimate} can be bounded by means of Theorem~\ref{theorem local h-best approximation plane waves}.
For the third term, Theorem~\ref{prop proj error} can be applied.
\end{proof}

\section{A numerical experiment} \label{section numerical results}
Having recalled that the implementation details and a wide number of numerical experiments will be presented in~\cite{TVEM_Helmholtz_num}, we test the convergence of the method~\eqref{HH problem weak discrete formulation} on a test case with domain
$\Omega=(0,1)^2$ and exact solution
\begin{equation} \label{exact solution}
u(\xbold) 
 = H_0^{(1)}(k \vert \xbold - \xbold_0 \vert), \quad \quad \xbold_0=(-0.25,0),
\end{equation}
where $ H_0^{(1)}$ denotes the $0$-th order Hankel function of the first kind, see \cite[Chapter 9]{AbramowitzStegun_handbook}.
Notice that $u$ 
 is analytic in $\Omega$.

We aim at studying the relative $L^2$ error
\begin{equation} \label{exact L2 error}
\frac{\Vert u - \un \Vert_{0,\Omega}} {\Vert u \Vert_{0,\Omega}},
\end{equation}
$\un$ being the solution to~\eqref{HH problem weak discrete
  formulation} when  approximating~$u$.

As $\un$ is not available in closed form, the quantity~\eqref{exact L2 error} is not computable.
We consider instead the following quantity:
\begin{equation} \label{computable error}
\frac{\Vert u - \Pi_\p u_h \Vert_{0,\Omega}}{\Vert u \Vert_{0,\Omega}},
\end{equation}
where we recall that~$\Pi_\p$, which is defined piecewise as the projector~$\Pip$ in~\eqref{projector aK}, is explicitly computable via the degrees of freedom.

Besides, we employ the following explicit local stabilization forms:
\begin{equation} \label{D-recipe}
\SE(\un, \vn) = \sigma \sum_{\e \in \EE} \sum_{\ell \in \Je} \aE (\Pip \varphi_{\e,\ell}, \Pip \varphi_{\e,\ell}) \, \dof_{\e,\ell} (\un) \overline{\dof_{\e,\ell} (\vn)} \quad \forall \un,\,\vn \in \VE,\; \forall \E \in \taun
\end{equation}
for some positive constant~$\sigma$.

The reason why we employ such~$\SE(\cdot, \cdot)$, rather than the more standard choice
\begin{equation} \label{dofi dofi}
\widetilde S^\E(\un, \vn) = \sum_{\e \in \EE} \sum_{\ell \in \Je} \dof_{\e,\ell} (\un) \overline{\dof_{\e,\ell} (\vn)} \quad \forall \un,\,\vn \in \VE,\; \forall \E \in \taun,
\end{equation}
introduced in~\cite{VEMvolley} for the Poisson problem, is that we aim at satisfying~\eqref{equations 34 and 35} with e.g.~$\alphah =1$.
In particular, given $\{\varphi_{\e,\ell}\}_{\e\in \EE,\, \ell\in \Je}$ the canonical basis~\eqref{local basis} on an element~$\E$, we want that
\[
\aE(\varphi_{\e,\ell}, \varphi_{\e,\ell}) \le \SE(\varphi_{\e,\ell}, \varphi_{\e,\ell})= \sigma\aE(\Pip \varphi_{\e,\ell}, \Pip \varphi_{\e,\ell}) \le \gamma_\h \Vert \varphi_{\e,\ell} \Vert_{1,\k,\E} ^2 \quad \forall \e\in \EE, \, \forall \ell\in \Je.
\]
In standard polynomial based VEM, a careful choice of the degrees of freedom guarantees that the energy of the basis functions scales like~$1$.
This is not the case in our setting, since the method hinges upon plane wave spaces.
Therefore, the standard choice~\eqref{dofi dofi} is corrected here by inserting the factor~$\aE(\Pip \varphi_{\e,\ell}, \Pip \varphi_{\e,\ell})$ as in~\eqref{D-recipe}, which mimics~$\aE(\varphi_{\e,\ell}, \varphi_{\e,\ell})$.
This approach is inspired by the so called diagonal recipe \cite{VEM3Dbasic, fetishVEM, fetishVEM} employed in the original VEM for the Poisson problem.
In the numerical experiments below, we fix~$\sigma=1$.

We highlight that the \ncTVEM, as presented in Section \ref{section definition local space}, suffers of strong ill-conditioning.
This is essentially because the basis functions become close to be linearly dependent as the edges of the
elements shrink or as the number of plane waves grows.
In order to overcome this drawback, we employ an edgewise or\-tho\-go\-na\-li\-za\-tion-and-fil\-te\-ring procedure.
Such procedure, described in detail in~\cite{TVEM_Helmholtz_num}, leads to a reduction of the number of degrees of freedom.
Here, we limit ourselves to mention that this strategy makes the method applicable for the same range of parameters as other plane wave methods; moreover, it speeds up the convergence rate of the method in terms of the number of degrees of freedom.

We test the method on a sequence of Voronoi-Lloyd meshes, obtained using
the algorithm in \cite{paulinotestnumericipolygonalmeshes}, with
decreasing mesh size; see Figure \ref{figure Voronoi meshes}. Note that such meshes are not nested by construction.
\begin{figure}[h]
\begin{center}
\includegraphics[width=0.4\textwidth]{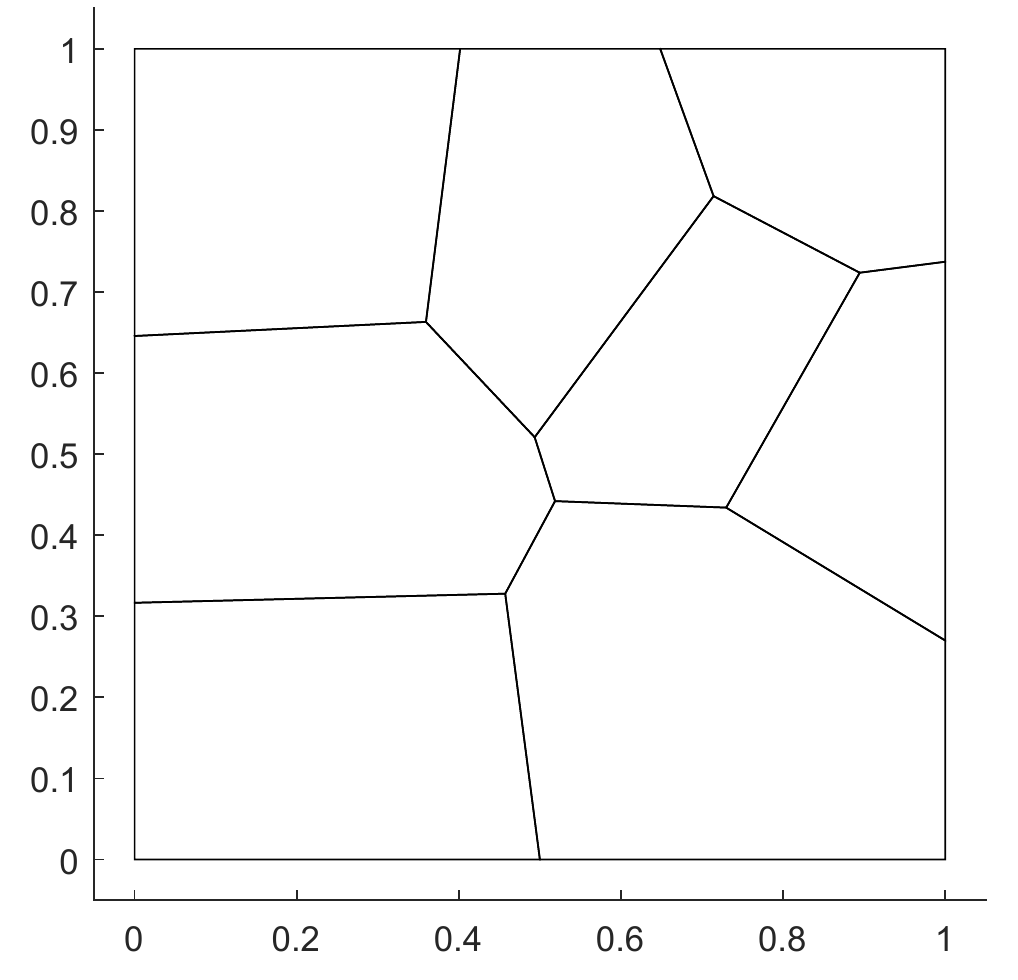}
\hspace{0.5cm}
\includegraphics[width=0.4\textwidth]{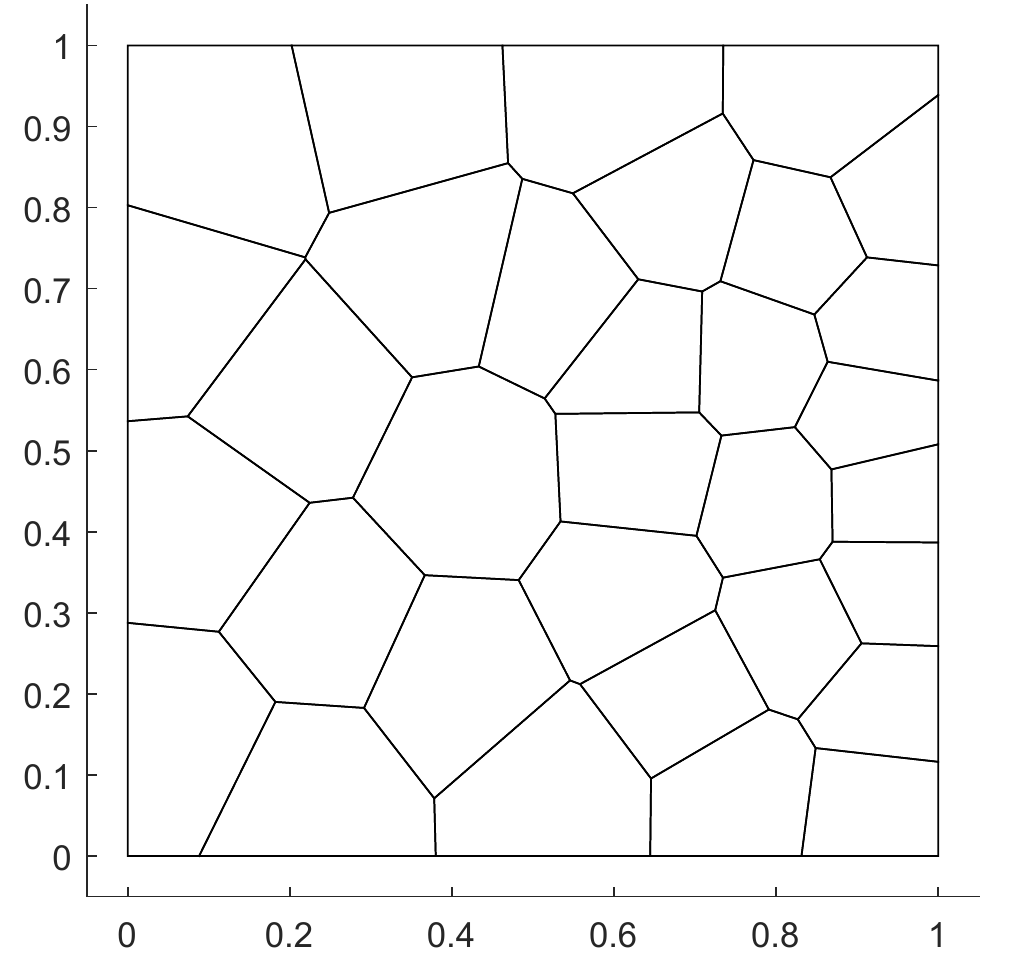}
\end{center}
\caption{Examples of Voronoi meshes with $8$ and $32$ elements.}
\label{figure Voronoi meshes} 
\end{figure}

In Figure \ref{figure numerical
  results 2}, 
we plot the computable quantity~\eqref{computable error}
obtained by employing $\p = 2q+1$ plane waves in the definition of the
local spaces $\PWE$ in~\eqref{plane wave bulk space}, with $q=4$ and 7, against the inverse of the product $\h\k$. 
We tested the method for different wave numbers~$\k$.

\begin{figure}[h]
\begin{center}
\begin{subfigure}[b]{0.48\textwidth}  
\includegraphics[width=1\textwidth]{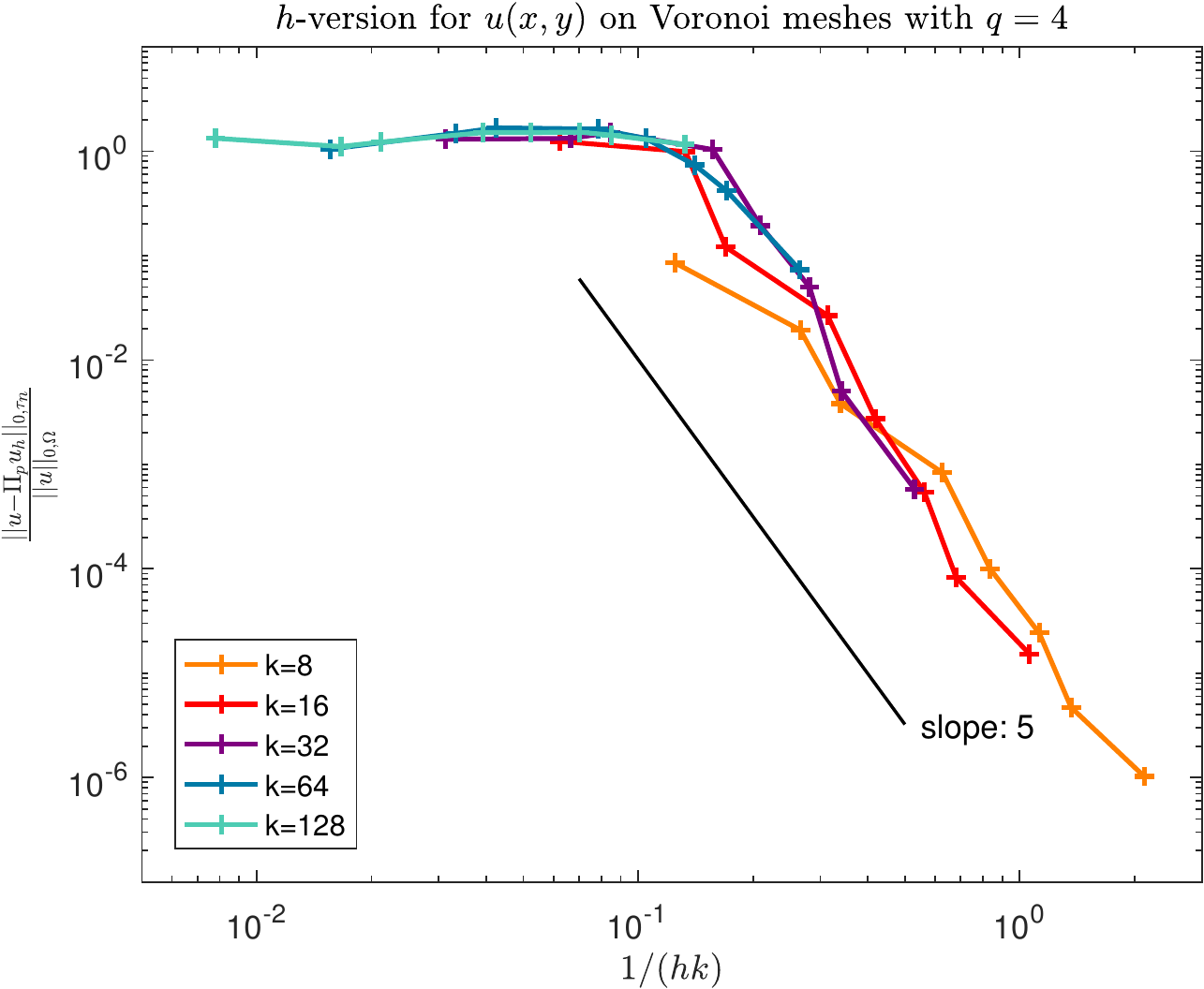}
\end{subfigure}
\hfill
\begin{subfigure}[b]{0.48\textwidth}   
\includegraphics[width=1\textwidth]{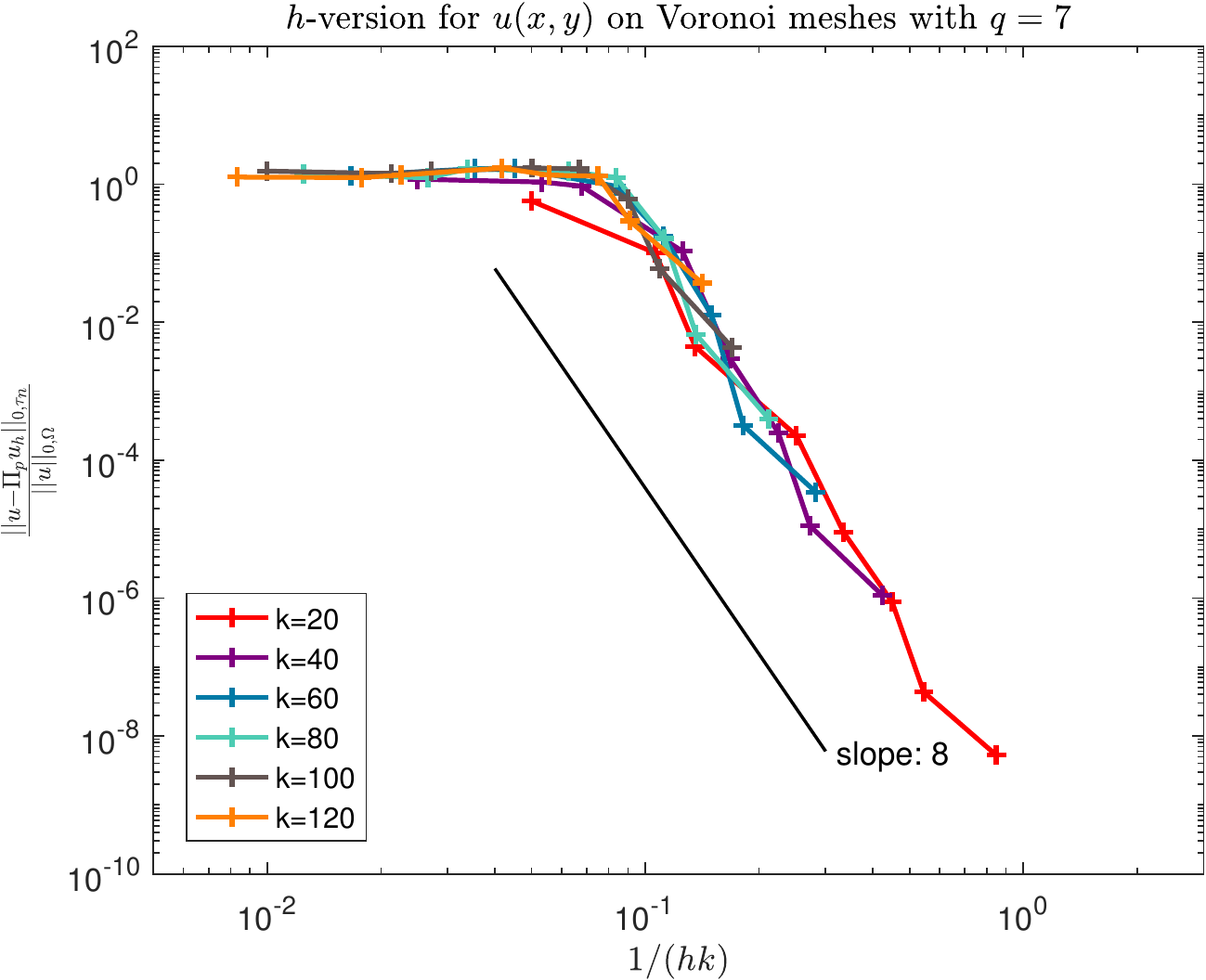}
\end{subfigure}
\end{center}
\caption{$\h$-convergence of the nonconforming Trefftz-VEM on the Voronoi-Lloyd meshes for different values of the wave number~$\k$.
The computable quantity~\eqref{computable error} is plotted against the inverse of the product $\h\k$. Local spaces with~$\p=9$ (left) and~$\p=15$ (right) plane waves are used.}
\label{figure numerical results 2}
\end{figure}

Firstly, we notice that the error curves are not precisely straight; this is due to the fact that the Voronoi
meshes (even after some Lloyd iterations) contain elements with substantially different size. We remark that similar tests on Cartesian meshes led to straight lines.
Moreover, the decreasing behavior stops once the product $\h\k$ becomes ``too small'', as compared to $\p$; this can be traced back to the ill-conditioning of the
plane wave basis (similar results are obtained employing plane wave-discontinuous Galerkin methods, see~\cite{GHP_PWDGFEM_hversion}).
The tests indicate algebraic convergence rate of order $\mathcal{O}((hk)^{q+1})$ (see also Figure \ref{figure numerical results 3}) and, as typical of plane wave based methods, a delayed onset of convergence for higher values of $k$.

\begin{figure}[h]
\begin{center}
\includegraphics[width=0.5\textwidth]{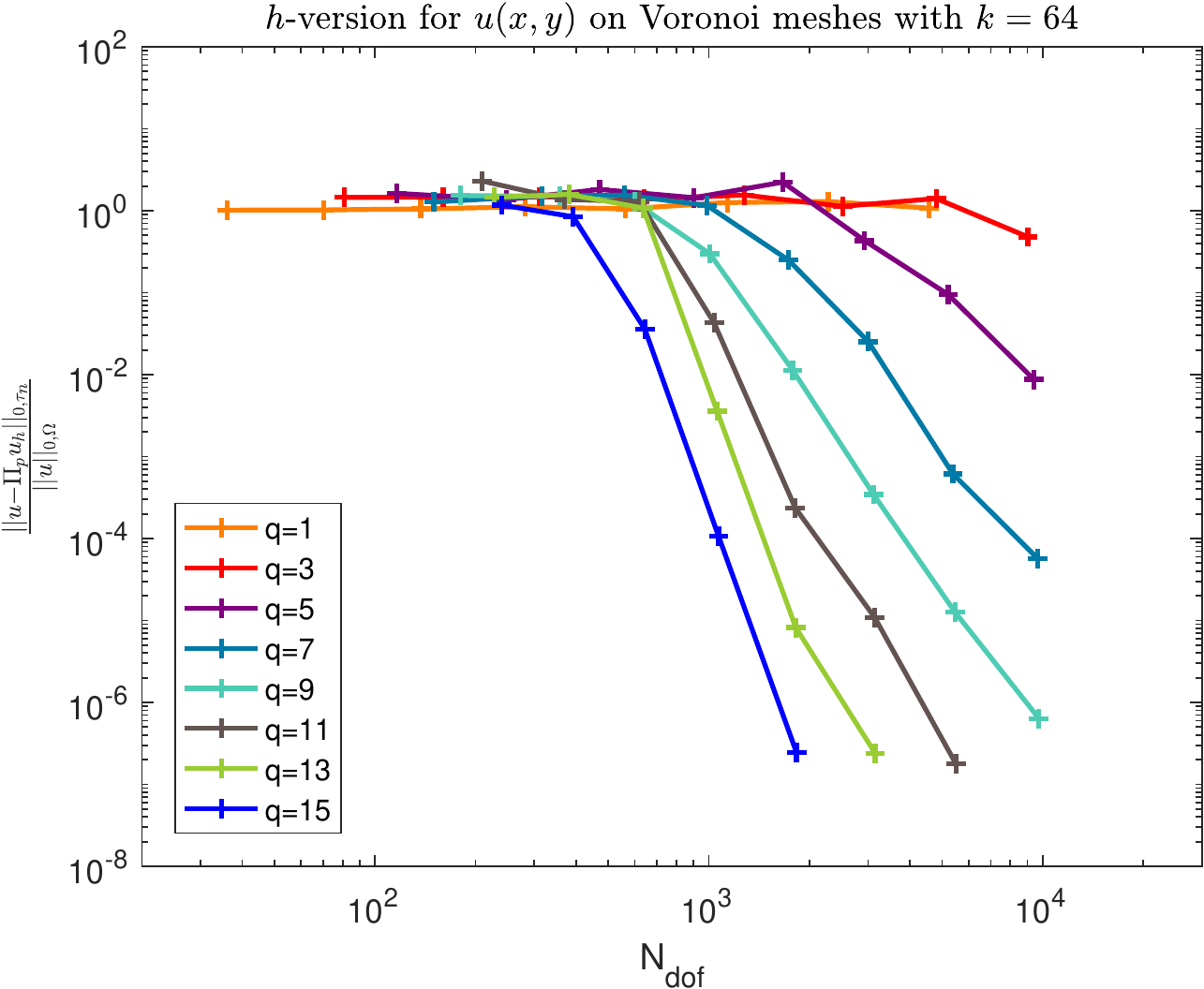}
\end{center}
\caption{$\h$-convergence of the nonconforming Trefftz-VEM on the Voronoi-Lloyd meshes for different values of $q$ and fixed wave number $k=64$. The computable quantity~\eqref{computable error} is plotted against the number of degrees of freedom.}
\label{figure numerical results 3}
\end{figure}

\section{Conclusions} \label{section conclusions}
In the present work, we have introduced a nonconforming Trefftz virtual element method based on polygonal grids for the two dimensional Helmholtz problem.
We do not hereby use fully discontinuous trial and test functions, but rather the jumps across interfaces are imposed to zero in a nonconforming sense.
In addition, local spaces are Trefftz, which implies that less degrees of freedom are needed for reaching a certain accuracy, as compared to standard polynomial based methods.
The definition of the \ncTVEM\,requires computable projectors, which map functions in the local approximation spaces into plane wave spaces, and computable stabilizations, which guarantee a discrete G\r arding inequality.
Importantly, only degrees of freedom on the mesh interfaces are used.

The construction of nonconforming Trefftz-VE spaces is based on the following strategy, which generalizes that of nonconforming harmonic VE spaces of \cite{conformingHarmonicVEM}, and which we deem can be extended to other (linear) settings:
\begin{enumerate}
\item on each element $\E\in \taun$, one introduces local discrete spaces $\VE$ made of implicitly defined functions in the kernel of the target differential operator and whose traces on each element edge are defined
such that $\VE$ contains a finite dimensional space, say, $W(\E)$, with good approximation properties and whose functions are available in closed form;
\item the degrees of freedom are defined so that the global trial and test spaces $V_\h$ can be built in a nonconforming fashion, and so that 
the best approximation error for functions in $V_\h$ is bounded by the best approximation error for functions in $\Pi_{\E\in \taun} W(\E)$.
\end{enumerate}
We underline that, differently from Trefftz discontinuous methods, the
\ncTVEM\,allows to recover information on the solution over the mesh skeleton, via edge projection operators.
Moreover, differently from standard partition of unity methods, its construction neither requires explicit knowledge of basis functions nor quadrature formulas.
The implementation of the \ncTVEM\,is described in~\cite{TVEM_Helmholtz_num}, together with an extensive discussion of its numerical performance, and a comparison with other methods in the literature.

\section*{Acknowledgements}
The authors have been funded by the Austrian Science Fund (FWF) through the project F 65 (L.M. and I.P.) and the project P 29197-N32 (I.P. and A.P.), and by the Vienna Science and Technology Fund (WWTF) through the project MA14-006 (I.P.).\\
The authors wish to thank Andrea Moiola (University of Pavia) for fruitful discussions on the improvement of the conditioning of plane wave based methods via a local reduction of degrees of freedom.

{\footnotesize
\bibliography{bibliogr}
}
\bibliographystyle{plain}

\end{document}